\documentclass{amsart}
\usepackage{amsmath}
\usepackage{amssymb}
\usepackage{amsfonts}
\usepackage{amsthm}
\usepackage{mathrsfs}
\usepackage[all]{xy}
\usepackage{ytableau}
\usepackage{color}
\usepackage{multicol}
\usepackage{enumerate}
\usepackage{enumitem}
\usepackage{comment}
\usepackage[vcentermath,enableskew]{youngtab}
\usepackage{multirow}
\usepackage{tikz-cd}
\usepackage{cancel}

\usetikzlibrary{arrows}

    \usetikzlibrary{decorations.markings, arrows.meta}

    \tikzset{
        mid arrow/.style={
            postaction={decorate,decoration={
                markings,
                mark=at position .6 with {\arrow[scale=1.2, #1]{Stealth}}
            }}
        }
    }

\numberwithin{equation}{section}

\theoremstyle{theorem}

\newtheorem{theorem}{Theorem}[section]
\newtheorem{corollary}[theorem]{Corollary}
\newtheorem{lemma}[theorem]{Lemma}

\newtheorem{proposition}[theorem]{Proposition}

\theoremstyle{definition}

\newtheorem{remark}[theorem]{Remark}
\newtheorem{definition}[theorem]{Definition}

\newtheorem{example}[theorem]{Example}

\newcommand{\sub}{\subseteq}

\newcommand{\sm}{\setminus}

\newcommand{\CC}{\mathbb{C}}

\newcommand{\ZZ}{\mathbb{Z}}

\newcommand{\FF}{\mathbb{F}}

\newcommand{\ff}{\mathcal{F}}

\newcommand{\mm}{\mathcal{M}}

\newcommand{\uu}{\mathcal{U}}

\newcommand{\LL}{\mathcal{L}}

\newcommand{\gl}{\mathfrak{gl}}
\newcommand{\sll}{\mathfrak{sl}}

\newcommand{\len}{\mathbf{len}}

\newcommand{\GL}{\operatorname{GL}}
\newcommand{\SL}{\operatorname{SL}}

\newcommand{\bfr}{\mathbf{r}}
\newcommand{\bfj}{\mathbf{j}}
\newcommand{\bfla}{\bs\lambda}

\newcommand*\circled[1]{\tikz[baseline=(char.base)]{
            \node[shape=circle,draw,inner sep=0.1pt] (char) {#1};}}
\newcommand{\ol}{\overline}
\newcommand{\ra}{\rightarrow}

\newcommand{\hra}{\hookrightarrow{}}

\newcommand{\onto}{\twoheadrightarrow}

\newcommand{\eps}{\epsilon}
\newcommand{\im}{\operatorname{Im}}
\newcommand{\End}{\operatorname{End}}

\newcommand{\bs}{\boldsymbol}

\newcommand{\uvsln}{U_v(\sll_n)}

\newcommand{\drcomm}[1]{\textcolor{red}{#1}}
\newcommand{\pgcomm}[1]{\textcolor{blue}{#1}}

\allowdisplaybreaks

\usepackage[colorlinks=true, pdfstartview=FitV, linkcolor=blue, citecolor=blue, urlcolor=blue]{hyperref}
\usepackage{graphicx} 

\title{Combinatorial description of the mirabolic quantum Schur-Weyl duality}
\author{Pallav Goyal}
\address{P.~Goyal: Centre for Quantum Mathematics, Department of Mathematics and Computer Science, University of Southern Denmark}
\email{pallavg@imada.sdu.dk}

\author{Daniele Rosso}
\address{D.~Rosso: Department of Mathematics and Actuarial Science, Indiana University Northwest}
\email{drosso@iu.edu}

\date{\today}

\begin{document}

\begin{abstract}
We prove an explicit correspondence between the irreducible finite dimensional representations of the mirabolic quantum group $MU(n)$ and those of the mirabolic Hecke algebra $MH_d$ given by the mirabolic quantum Schur-Weyl duality. In particular, this resolves a generalization of a conjecture by the second author. The proof is via a careful analysis of the $MU(n)$-weight spaces and the eigenvalues of the Jucys Murphy elements of $MH_d$ on the mirabolic tensor space.
\end{abstract}

\maketitle

\section{Introduction}

One way of studying representations of the general linear group $\GL_n=\GL(V)$ (or the special linear group $\SL(V)$) is by taking advantage of its Schur-Weyl duality with the symmetric group $\mathcal{S}_d$. We have a left diagonal action of $\GL(V)$ on the tensor product $V^{\otimes d}$, and a right action of $\mathcal{S}_d$ by permuting the tensor factors. Since the two actions satisfy the double centralizer property, we can use knowledge of the representations of $\mathcal S_d$ to gain information about the representations of $\GL(V)$. This was originally done in the work of Schur \cite{Sch01}, which was popularized by Weyl \cite{Weyl39}. This idea can be extended to their respective quantum versions: the quantized enveloping algebra of the Lie algebra $\gl_n$ on one side, and the Iwahori-Hecke algebra $H_d$ on the other side. This quantum Schur-Weyl duality was first described by Michio Jimbo in \cite{J86}.

A geometrical construction of the quantum Schur algebras, which are finite dimensional quotients of the quantized enveloping algebra of the Lie algebra $\gl_n$ (or $\sll_n$), was given by  Beilison, Lusztig and MacPherson in \cite{BLM}. They showed that the quantum Schur algebra is isomorphic to the space of $\GL_d$-invariant functions over the variety of pairs of \emph{$n$-step partial flags} in a $d$-dimensional vector space over the finite field $\FF_q$ with a convolution product.

Given that the Iwahori-Hecke algebra $H_d$ can be realized as the  space of invariant functions on the variety of pairs of \emph{complete} flags in $\FF_q^d$ (see \cite{Iwa64}), the quantum Schur-Weyl duality can then be interpreted as the convolution action on the space of invariant functions on pairs of one partial and one complete flag, as was done by Grojnowski and Lusztig in \cite{GL}.

All the above mentioned definitions can be extended to the `mirabolic' setting, as follows. Instead of considering pairs of flags (partial or complete), we can consider triples of two flags and a vector in $\FF_q^d$, where a convolution product was defined by the second author. By working with complete flags, this construction gives the mirabolic Hecke algebra, whose representations were studied in \cite{R14}. Analogously to \cite{BLM} and \cite{GL}, by working with triples of two $n$-step partial flags and a vector (resp. one partial flag, one complete flag, and a vector) we obtain mirabolic versions of the quantized enveloping algebra and a mirabolic quantum Schur-Weyl duality. These were first defined in \cite{R15}, where, in particular, a complete classification of the finite dimensional representations for $MU(n)$ (the mirabolic quantum group version of $\sll_n$) was obtained when $n=2$. This work was generalized by the authors in \cite{GR26} to classify representations for $MU(n)$  for arbitrary $n$.

In \cite[Theorem 6.9]{R15} there was also a first step towards a combinatorial description of the mirabolic quantum Schur-Weyl duality by decomposing the mirabolic tensor space as a left $MU(2)$-module. Some of the results in \cite{R15} were expanded upon by Fan, Zhang and Ma in \cite{FZM}, where they found 
explicit formulae for the left and right action involved in the mirabolic quantum Schur-Weyl duality. 

In this paper, 
we give a complete description of the mirabolic quantum Schur-Weyl duality in terms of the irreducible representations of the mirabolic Hecke algebra and the mirabolic quantum group, in particular proving a generalization for any $n$ of Conjecture 6.8 of \cite{R15}.

A summary of our main results and the organization of the paper is as follows:
\begin{itemize}
    \item In Section~\ref{sec:back}, we we discuss the relevant background on the mirabolic quantum group and the mirabolic Hecke algebra. In particular, we recall their constructions as convolution algebras as well as the classifications of their finite dimensional irreducible representations.
    \item In Section \ref{sec:MSWD}, we state our main result (Theorem \ref{thm:MSWD}) giving an explicit description of the correspondence between the simple modules of $MU(n)$ and those of the mirabolic Hecke algebra under the mirabolic quantum Schur-Weyl duality. We also prove some special cases of this result in Sections~\ref{sec:MSWDs=0} and~\ref{sec:leftmod}, which we will be needing for the general proof.
    \item Finally, in Section~\ref{sec:MSWDproof} we complete the proof of Theorem~\ref{thm:MSWD}. The main tool used is an explicit construction in Section~\ref{sec:zeroeigen} of vectors in the mirabolic tensor space $MV_{n,d}$ that are simultaneusly annihilated by the Jucys-Murphy elements of $MH_d$.
\end{itemize}

\subsection{Future Directions}
Our work raises several interesting questions, which can be the basis for future research.
\begin{itemize}
\item There has been a substantial interest in the study of the quantum group $\uvsln$, the Iwahori-Hecke algebra $H_d$ and their representations from the point of view of diagrammatics and categorification. This naturally motivates the investigation of similar questions for the algebras $MU(n)$ and $MH_d$. Also of interest might be the study of the mirabolic analog of the Temperley-Lieb algebra (see Remark~\ref{rem:TL}).
\item Another potential direction is to find explicit combinatorial descriptions of the isotypic components of the $MU(n)$ and $MH_d$-representations in the mirabolic tensor space $MV_{n,d}$, like the ones that exist in the usual quantum tensor space, in terms of Young symmetrizers. A starting point for this could be the basis computed in Lemma~\ref{lem:zerobasis}.
\item There exists a Schur-Weyl duality between the Ariki-Koike cyclotomic Hecke algebra and shifted quantum affine algebras (see \cite{WAD24}). Since the mirabolic Hecke algebra $MH_d$ is a quotient of the Ariki-Koike algebra, it would be interesting to explore connections between \cite{WAD24} and our results.
\item Apart from the usual Schur-Weyl duality, the representations of the Lie algebras $\sll_n$ and $\sll_m$ for $m,n\in \ZZ_{> 0}$ also satisfy a certain Howe duality with each other. In \cite{fan_mirabolic_2026}, the authors proved a mirabolic version of this Howe duality for the algebras $MU(n)$ and $MU(m)$. Using the classification theorem of \cite{GR26} and the techniques developed in this paper, it should be possible to find a combinatorial description of the results of \cite{fan_mirabolic_2026}.
\end{itemize}

\subsection*{Acknowledgments}
P.G. was partly supported in this research by the Villum Ascending Investigator grant no. $84071$ from Villum Fonden. D.R. was supported by an AMS-Simons Research Enhancement Grants for PUI faculty. The authors would like to thank Zhicheng Zhang for graciously sharing updated versions of the formulae in \cite{FZM}.

\section{Background on the mirabolic quantum group and the mirabolic Hecke algebra} \label{sec:back}

\subsection{Mirabolic quantum group}

We recall some definitions and results 
about the mirabolic quantum group $MU(n)$ from \cite{GR26}. We work over the field $K=\CC(v)$, where $v$ is an indeterminate, and we let $n$ be a positive integer.  Let $A=(a_{ij})$ be the Cartan matrix of type $A_{n-1}$, that is, for all $1\leq i,j,\leq n-1$ we have
$$a_{ij}=\begin{cases} 2 & \text{ if }i=j, \\ -1 & \text{ if }|i-j|=1, \\ 0 & \text{ if }|i-j|>1.\end{cases}$$

\begin{definition}\label{def:uvsln}
The \emph{mirabolic quantum group}, denoted by $MU(n)$, is the unital associative algebra over $K$ with generators $\ell$, and $e_i, f_i, k_i^{\pm 1}$ for $1\leq i\leq n-1$, subject to the relations:
\begin{align}\label{eq:usln.1} k_i^{\pm1}k_i^{\mp1}=1, & \quad k_ik_j=k_jk_i \\
\label{eq:usln.2} k_ie_jk_i^{-1} = v^{a_{ij}}e_j, & \quad k_if_jk_i^{-1} = v^{-a_{ij}}f_j\\
\label{eq:usln.3} e_if_j-f_je_i & =  \delta_{ij}\frac{k_i-k_i^{-1}}{v-v^{-1}}\\
\label{eq:usln.4} e_ie_j=e_je_i, & \quad f_if_j=f_jf_i \quad\text{ if }|i-j| \neq 1 \\
\label{eq:usln.5} (v+v^{-1})e_ie_je_i & =  e_i^2e_j + e_je_i^2 \quad\text{ if }|i-j| = 1\\
\label{eq:usln.6} (v+v^{-1})f_if_jf_i & =  f_i^2f_j + f_jf_i^2 \quad\text{ if }|i-j| = 1.\\
\label{eq:mun.1} \ell^2=\ell,&\quad k_i\ell=\ell k_i\\
\label{eq:mun.2} \ell e_i=e_i\ell,&\quad \ell f_i=f_i\ell \quad\text{ if }i\geq 2\\
\label{eq:mun.3} \ell e_1\ell=\ell e_1,&\quad \ell f_1\ell=f_1\ell\\
\label{eq:mun.4} (v+v^{-1})e_1\ell e_1 &= v^{-1}e_1^2\ell + v\ell e_1^2\\
\label{eq:mun.5} (v+v^{-1})f_1\ell f_1 &= v^{-1}\ell f_1^2 + vf_1^2\ell.\end{align}
\end{definition}

The subalgebra of $MU(n)$ generated by $e_i, f_i, k_i^{\pm1}$ (with relations \eqref{eq:usln.1}-\eqref{eq:usln.6}) is precisely the quantized enveloping algebra $\uvsln$ of the Lie algebra $\mathfrak{sl}_n$.

\begin{proposition} [{\cite[Proposition $2.7$]{GR26}\label{prop:comodule}}]
The algebra $MU(n)$ is a co-module algebra over $\uvsln$ via the co-module map:
\begin{align*}
\rho:MU(n)\to &\uvsln\otimes MU(n)\\
e_i\mapsto 1\otimes e_i+e_i\otimes k_i, &\qquad f_i\mapsto f_i\otimes 1 + k_i^{-1}\otimes f_i,\\
k_i^{\pm1} \mapsto k_i^{\pm1}\otimes k_i^{\pm1}, &\qquad \ell\mapsto 1\otimes \ell.
\end{align*}
\end{proposition}

All modules that we work with in this paper are assumed to be finite dimensional (over $K$) left modules. We recall from \cite[\S 2.2.1]{GR26} that all such modules for $MU(n)$ are \emph{weight modules}, and we will further restrict ourselves to type $\mathbf{1}$, which means that for any module $M$, we have a direct sum decomposition into mirabolic weight spaces
\[M=\bigoplus_{\lambda,\epsilon}M_{\lambda,\epsilon},\]
\[M_{\lambda,\epsilon}:=\{x\in M~|~\ell \cdot x=\epsilon x,~k_i\cdot x=v^{\lambda_i-\lambda_{i+1}}x,\quad\forall i=1,\ldots,n-1\}\]
where $\epsilon\in\{0,1\}$ (since $\ell$ is an idempotent) and $\lambda$ is a weight for $\uvsln$, that is $\lambda=(\lambda_1,\cdots,\lambda_n)\in \ZZ^n/\sim$ with the equivalence relation given by $(\lambda_1,\cdots,\lambda_n)\sim(\lambda_1+t,\cdots,\lambda_n+t)$ for all $t\in\ZZ$.

\begin{definition}[{See \cite[Def 2.13-Lem 2.16]{GR26}}]\label{def:W_r} For any $0\leq r\leq n$ we define the set $[n]_r:=\{I\subset\{1,\ldots,n\}~|~\#I=r\}$, and the module $W_r=K\mathrm{-span}\{u_I~|~I\in[n]_r\}$
with the action given by   
\[ k_i\cdot u_I=\begin{cases}v u_I & \text{ if }i\in I,~ i+1\not\in I,\\
v^{-1} u_I & \text{ if }i\not\in I,~ i+1\in I,\\
u_I & \text{ otherwise,}\end{cases}\quad e_i\cdot u_I=\begin{cases}u_{I\setminus\{i+1\}\cup\{i\}} & \text{ if }i+1\in I,~ i\not\in I,\\
0 & \text{ otherwise,}\end{cases}\]
\[ f_i\cdot u_I=\begin{cases}u_{I\setminus\{i\}\cup\{i+1\}} & \text{ if }i\in I,~ i+1\not\in I,\\
0 & \text{ otherwise,}\end{cases}\qquad \ell\cdot u_I=\begin{cases} u_I & \text{ if }1\not\in I,\\
0 & \text{ if }1\in I.   
\end{cases}\]
\end{definition}

Viewed as a module over the subalgebra $\uvsln\sub MU(n)$, the space $W_r$ is precisely the simple representation whose highest weight is the fundamental dominant integral weight $\omega_r:=(1,1,\dots, 1,0, 0, \dots, 0)=(1^r0^{n-r})$.

\begin{theorem}[{\cite[Thm 5.1, Thm 5.3]{GR26}}]
The category of finite dimensional $MU(n)$-representations is semisimple. A complete list of isomorphism classes of finite dimensional simple $MU(n)$ modules is given by $L_{(\lambda,1^r)}:=L_{\lambda}\otimes W_r$, where $0\leq r\leq n$, $\lambda$ is a dominant integral weight (i.e. $\lambda_1\geq\lambda_2\geq\cdots\geq\lambda_n$), and $L_\lambda$ is the unique simple module for $\uvsln$ with highest weight $\lambda$. The $MU(n)$-module structure comes from the comodule map $\rho$.\end{theorem}

Notice that $L_{(\lambda,1^r)}$ was denoted by $L_{\lambda,r}$ in \cite{GR26}.

\begin{remark}\label{rem:pi0-pi1}
There are two surjective maps $\pi_0,\pi_1:MU(n)\to\uvsln$, where the generators $k_i^{\pm 1},e_i,f_i$ map to the ones with the same name, and $\pi_\epsilon(\ell)=\epsilon$, for $\epsilon=0,1$. Then, we have that, for a simple highest weight $\uvsln$-module $L_\lambda$, $\pi_0^*(L_\lambda)\simeq L_{(\lambda,1^n)}$ and $\pi_1^*(L_\lambda)\simeq L_{(\lambda,\emptyset)}$.
\end{remark}

\begin{proposition}[{\cite[Prop 5.4]{GR26}}] \label{prop:weightspace}
Let $M$ be a finite dimensional $MU(n)$-module. Then its decomposition as a direct sum of simples is uniquely determined by the dimensions of its mirabolic weight spaces $M_{\lambda,\epsilon}$.
\end{proposition}
\begin{definition}
Let $M$ be a module for $MU(n)$, and let $x\in M$ be a weight vector. For $0\leq r\leq n$, we say that $x$ has \emph{depth} $r$ if 
$$\ell f_1\cdots f_j \cdot x= 0\qquad \text{ for all }0\leq j<r,$$ 
and
$$\ell f_1f_2\cdots f_r\cdot  x\neq 0 \quad \text{ (this condition doesn't apply when $r=n$).}$$
Analogously, we say that $x$ has \emph{co-depth} $r$ if 
$$(1-\ell) e_1\cdots e_j \cdot x= 0\qquad \text{ for all }0\leq j<r,$$ 
and
$$(1-\ell) e_1e_2\cdots e_r\cdot  x\neq 0 \quad \text{ (this condition doesn't apply when $r=n$).}$$
If $M\neq 0$, we define the (co-)depth of $M$ to be the maximum of the (co-)depth of $x$, where $x$ varies among all nonzero weight vectors in $M$.
\end{definition}
\begin{remark}
The $0$ vector in any module has depth $n$ and co-depth $n$.
\end{remark}
\begin{example}\label{ex:(co-)depth} 
Let $\lambda$ be a dominant weight and $0\leq r\leq n$. Pick any weight vector $y\in L_{\lambda}$ and $I\in [n]_r$. Then, using the action of $MU(n)$ on $W_r$ from Definition \ref{def:W_r} and the comodule map $\rho$ from Proposition \ref{prop:comodule}, it is an easy computation to check that the vector $y\otimes u_I$ is a weight vector of $L_{(\lambda,1^r)}$ whose depth is $s$, where $s$ is the maximum number such that $\{1,2,\dots, s\}\subset I$, and whose co-depth is $s'$, where $s'$ the maximum number such that $\{1,2,\dots, s'\}\cap I=\emptyset$. In particular, the module $L_{(\lambda,1^r)}$ has depth $r$ and co-depth $n-r$.
\end{example}

The following lemma provides the relationship between depth and co-depth.

\begin{lemma} \label{lem:oldepth}
Let $M$ be any finite dimensional $MU(n)$-module that is a direct sum of simple modules of depth $r$.
\begin{enumerate}
\item We have the equality:
\[\textrm{depth}(M)+{\textrm{co-depth}}(M)=n.\]
\item Define $M'\sub M$ (resp. $\ol{M'}\sub M$) to be the subspace spanned by all weight vectors of depth $r$ (resp. co-depth $n-r$). Then, there is a linear isomorphism:
\[\psi:\ol{M'}\to M'\]
such that for any $w\in \ol{M'}$ with weight $\mu$, the vector $\psi(w)$ has weight $\mu+\omega_r-\omega_r'$, where $\omega_r'$ is the weight $(0,0,\dots, 0, 1, 1,\dots 1)=(0^{n-r}1^r)$. Furthermore, there exists an element $m\in MU(n)$ (depending upon $w$) such that $\psi(w)=m\cdot w$.
\end{enumerate}
\end{lemma}

\begin{proof}
It suffices to prove the lemma when $M=L_{(\lambda,1^r)}$ is a simple $MU(n)$-module. The first assertion follows from that fact that $\textrm{depth}(L_{(\lambda,1^r)})=r$ and $\textrm{co-depth}(L_{(\lambda,1^r)})=n-r$, as explained in Example \ref{ex:(co-)depth}.

Let $I=\{1,2,\dots,r\}$ and $\ol{I}=\{n-r+1,n-r+2,\dots, n\}$. Then, it is clear that $M'=L_{\lambda}\otimes\{u_I\}$ and $\ol{M'}=L_{\lambda}\otimes\{u_{\ol{I}}\}$. Then, the required linear map $\psi$ is given by:
\[w\otimes u_{\ol I} \mapsto w\otimes u_{I}\]
for any $w\in L_{\lambda}$.
Furthermore, since $M$ is a simple $MU(n)$-module, there exists an $m\in MU(n)$ (depending upon $w$) such that $m\cdot(w\otimes u_{\ol I}) = w\otimes u_{I}$.
\end{proof}

\begin{remark}
Part$(1)$ of Lemma \ref{lem:oldepth} is not true for arbitrary $MU(n)$-modules. For example, if $M=L_{(\lambda,1^0)}\oplus L_{(\lambda', 1^n)}$, we have $\textrm{depth}(M)=\textrm{co-depth}(M)=n$.
\end{remark}

\subsection{Mirabolic Hecke algebra}
We consider the mirabolic Hecke algebra $MH_d$, in the presentation given in \cite{R14}.

\begin{definition}The \emph{mirabolic Hecke algebra} $MH_d$ is defined as the unital $K$-algebra generated by elements $e,\tau_1,\ldots, \tau_{d-1}$ satisfying the following relations:
\begin{align} \label{nrel1} \tau_i^2 & = (v^2-1)\tau_i+v^2 \\
\label{nrel2} \tau_i\tau_{i+1}\tau_i& =\tau_{i+1}\tau_i\tau_{i+1}\\
\label{nrel3} \tau_i\tau_j&=\tau_j\tau_i & |i-j|\geq 2 \\
\label{nrel4}e^2&=e & \\
\label{nrel5}e\tau_i&=\tau_ie & i\geq 2 \\
\label{nrel6}e\tau_1e\tau_1&=\tau_1e\tau_1e & \\
\label{nrel7}e\tau_1e\tau_1& =\tau_1e\tau_1-e\tau_1e+\tau_1e+e\tau_1+e-\tau_1-1. &
\end{align}
\end{definition}
In comparison with the presentation used in \cite{FZM} (where the above algebra is denoted by $\mathbf{MH}$), we have $e=v^{-2}(\tau_0+1)$.
The Iwahori-Hecke algebra $H_d$ of the symmetric group $\mathcal{S}_d$ is isomorphic to the $K$-algebra generated by the elements $\tau_1,\ldots, \tau_{d-1}$ and satisfying the relations~\eqref{nrel1},~\eqref{nrel2} and~\eqref{nrel3}. Hence, it can be viewed as a subalgebra of $MH_d$.

If $d\in\mathbb{Z}_{\geq 0}$, the $n$-tuple $\lambda=(\lambda_1,\ldots,\lambda_n)$ with $\lambda_i\in\mathbb{Z}_{\geq 0}$ is a \emph{partition} of $d$ if $\lambda_1\geq\lambda_2\geq\cdots\geq\lambda_n$ and $\sum_i\lambda_i=d$. We write $|\lambda|=d$ or $\lambda\vdash d$. We denote repeated parts by $a^b=(a,\ldots,a)$ where $a$ appears $b$ times, and we will denote the partition of $0$ by $\emptyset$. A \emph{bipartition} of $d$ is a pair of partitions $\bs{\lambda}=(\lambda^{(1)},\lambda^{(2)})$ such that $|\lambda^{(1)}|+|\lambda^{(2)}|=d$. The set of partitions of $d$ parametrizes the irreducible representations of the Iwahori-Hecke algebra $H_d$.

Next, we recall some facts about the representation theory of $MH_d$.

\begin{theorem} [{\cite[Propostion 4.11, Theorem 4.12]{R14}}] \label{thm:MHd-reps}
The category of finite dimensional representations of the algebra $MH_d$ is semisimple and its irreducible representations, denoted by $M^{\bs{\lambda}}$, are parametrized by bipartitions $\bs{\lambda}=(\lambda^{(1)},\lambda^{(2)})$ such that $|\lambda^{(1)}|+|\lambda^{(2)}| = d$, and $\lambda^{(2)}=1^s$ for some $s\geq 0$. 
\end{theorem}

\begin{remark}\label{rem:p-proj}
\begin{enumerate}
\item There is surjective algebra map $p:MH_d\to H_d$, with $p(e)=1$ and $p(\tau_i)=\tau_i$. If $M^{\lambda}$ is the simple $H_d$-representation corresponding to a partition $\lambda\vdash d$, we have:
\[M^{(\lambda,\emptyset)}=p^*(M^{\lambda}).\]
\item There exists an anti-automorphism $MH_d\to MH_d$ that is induced by sending each generator to itself. As a result, there is an equivalence between the categories of left $MH_d$-modules and right $MH_d$-modules.
\end{enumerate}
\end{remark}
We can explicitly describe bases for these irreducible representations as follows. For a partition $(\lambda_1,\ldots,\lambda_n)\vdash d$, we can draw a Young diagram with $\lambda_i$ boxes in row $i$, for each $i=1,\ldots,n$. (We will use the English (top-left justified) convention for Young diagrams.) A Young diagram where the boxes are filled with numbers is called a \emph{tableau}.
\begin{definition}
\begin{enumerate}
\item Given a bipartition $\bs{\lambda}=(\lambda^{(1)},\lambda^{(2)})$ of $d$, a \emph{standard bitableau} $U=(U^1,U^2)$ of shape $\bs{\lambda}$ is a pair such that $U^i$ is a tableau of shape $\lambda^i$ where the boxes of the two tableaux are filled 
with the numbers from $1$ through $d$ such that entries increase in each row from left to right and in each column from top to bottom in each tableau.
\item Consider a standard bitableau $U=(U^1,U^2)$ and an integer $1\leq i\leq d$. We define the \emph{content}:
\[c_U(i)=j-1+c_{U^j}(i).\]
Here, $j$ is such that the box containing $i$ is in the tableau $U^j$, and $c_{U^j}(i)=m-l$, where $i$ lies in the $m^{th}$ column and $l^{th}$ row of $U^j$.
\end{enumerate}
\end{definition}
\begin{example}
This is an example of a standard bitableaux with $d=9$ and shape $((2,2,1),(3,1))$
\[U=\left(\young(14,37,8),\young(259,6)\right).\]
In this case we have $c_U(2)=1$, $c_U(8)=-2$.
\end{example}
Let $T(\bs{\lambda})$ denote the set of standard bitableaux of shape $\bs{\lambda}=(\lambda,1^s)$, with $|\lambda|+s=d$. Then, as a vector space,
\[M^{\bs{\lambda}}=K\mathrm{-span}\{x_U:U=(U^1,U^2)\in T(\bs{\lambda})\}.\]
In particular, if $f^{\lambda}$ denotes the number of standard tableaux of shape $\lambda$, then $\dim(M^{\bs\lambda})={d\choose s}f^{\lambda}$.

\begin{remark} \label{rem:hecke-basis}
With reference to Remark~\ref{rem:p-proj}, for any partition $\lambda\vdash d$, the representation $M^{\lambda}$ of the Hecke algebra $H_d$ has a basis given by $x_U$, where $U$ is a standard tableau of shape $\lambda$, such that $x_{(U,\emptyset)}=p^*(x_U)$.
\end{remark}

\begin{definition}
The \emph{mirabolic Jucys-Murphy elements} $\{\LL_i:i=1,\ldots,d\}$ of $MH_d$ are defined by
$$\LL_i:=v^{2-2i}\tau_{i-1}\cdots\tau_1e\tau_1\cdots\tau_{i-1}.$$
(In particular, $\LL_1:=e$.)
\end{definition}

These elements are important because the center of the algebra $MH_d$ is exactly given by the ring of symmetric functions in the mirabolic Jucys-Murphy elements (see \cite[Theorem 6.4]{R14}), and because their action characterizes the irreducible representations.

\begin{theorem} [{\cite[Corollary 6.3]{R14}}]\label{thm:JM_action}
On the $MH_d$-representation $M^{\bs{\lambda}}$, the action of the $\LL_i$'s is given via:
\[\LL_i\cdot x_U = \eps_iv^{2c_U(i)}x_U, \qquad\textrm{where } \eps_i=\begin{cases}1 \textrm{ if } i\in U^1,\\0\textrm{ if } i\in U^2,\end{cases}\]
for any $1\leq i\leq d$ and $U\in T(\bs{\lambda})$.
\end{theorem}

\begin{remark} \label{rem:simulbases}
\begin{enumerate}
\item The above theorem implies that the action of the Jucys-Murphy elements on any representation $M^{\bs\lambda}$ is simultaneously diagonalizable with respect to the basis $\{x_{U}: U\in T(\bs\lambda)\}$. Consequently, subsets of these basis act as bases for (simultaneous) eigenspaces of the Jucys-Murphy elements.

\item We note that the element $\LL_i$ is sent to itself by the anti-automorphism mentioned in Remark~\ref{rem:p-proj}(2). Therefore, the above theorem is valid even if $M^{\bs\lambda}$ is viewed as a right $MH_d$-representation via the anti-automorphism.
\end{enumerate}
\end{remark}




\subsection{Convolution}\label{sec:conv}
The mirabolic quantum group and the mirabolic Hecke algebra arise from a convolution construction. We recall some of the basic notions here and we refer the reader to \cite[\S 3]{R14}, \cite[\S 2]{R15} and \cite[\S 4]{FZM} for more details.

Let $\FF_q$ be the finite field with $q$ elements, $G:=\GL(\FF_q^d)=\GL_d(\FF_q)$, $\ff(d)$ the variety of complete flags in $\FF_q^d$, and $\ff(n,d)$ the variety of $n$-step partial flags in $\FF_q^d$. More explicitly the flag varieties can be thought of as the projective varieties consisting of increasing sequences of subspaces:
\[\ff(d)=\{F=(0=F_0\subsetneq F_1\subsetneq \cdots\subsetneq F_d=\FF_q^d)~|~\dim(F_i)=i\},\]
\[\ff(n,d)=\{F=(0=F_0\subseteq F_1\subseteq \cdots\subseteq F_n=\FF_q^d)\}.\]

The natural action of $G$ on $\FF_q^d$ induces an action on $\ff(d)$ and $\ff(n,d)$, and the diagonal action of $G$ on $\ff(d)\times\ff(d)$ (resp. $\ff(n,d)\times\ff(n,d)$) 
has finitely many orbits (independently of $q$) that can be parametrized by permutation matrices (resp. matrices with non-negative integer entries and total sum $d$). Then, the space of $G$-invariant $\CC$-valued functions over pairs of complete flags $\CC\left(\ff(d)\times\ff(d)\right)^{G}$,
has a natural convolution product given by the formula:
\begin{equation} \label{eq:convo1}
(\alpha*\beta)(F,F'):=\sum_{F''\in\ff(d)}\alpha(F,F'')\beta(F'',F'),
\end{equation}
for any $\alpha,\beta\in\CC\left(\ff(d)\times\ff(d)\right)^{G}$. The structure constants for the above product are polynomial in $q$. Therefore, the above convolution algebra can be considered the specialization of a $K(=\CC(v))$-algebra at $v=\sqrt{q}$. This $K$-algebra is exactly the Hecke algebra $H_d$. In terms of $\CC$-valued functions, the identity of $H_d$ is the characteristic function of the $G$-orbit $\{(F,F')\in\ff(d)\times\ff(d) ~|~F_i=F'_i \text{ for all }i=1,\ldots,d\}$, and the generator $\tau_j$ corresponds to the characteristic function of the $G$-orbit $\{(F,F')\in\ff(d)\times\ff(d) ~|~F_i=F'_i \text{ for }i\neq j,~F_j\cap F'_j=F_{j-1}=F'_{j-1}\}$.

Analogously, using the same formula of \eqref{eq:convo1}, but summing over $\ff(n,d)$, we get the convolution algebra $\CC\left(\ff(n,d)\times\ff(n,d)\right)^{G}$, which turns out to be the specialization at $v=\sqrt{q}$ of the quantum Schur algebra $U(n,d)$, which is a finite dimensional quotient of $\uvsln$ (see \cite{BLM}). Here, we can see the generators as follows: 
$$k_i^{\pm 1}=\sum_{\substack{(d_1,\ldots,d_n)\in\ZZ_{\geq0}^n\\\sum_i d_i=d}}v^{\pm(d_i-d_{i+1})}\mathbf{1}^{d_1,\ldots,d_n},$$
$$ e_i=\sum_{\substack{(d_1,\ldots,d_n)\in\ZZ_{\geq0}^n\\\sum_i d_i=d}}e_i^{d_1,\ldots,d_n}, \qquad f_i=\sum_{\substack{(d_1,\ldots,d_n)\in\ZZ_{\geq0}^n\\\sum_i d_i=d}}v^{1-d_{i+1}}f_i^{d_1,\ldots,d_n}$$
where $\mathbf{1}^{d_1,\ldots,d_n}$ is the characteristic function of the orbit $\{(F,F')\in\ff(n,d)\times\ff(n,d) ~|~F_j=F'_j,~\dim(F_j)=d_j \text{ for all }j=1,\ldots,n-1\}$, $e_i^{d_1,\ldots,d_n}$ is the characteristic function of the orbit $\{(F,F')\in\ff(n,d)\times\ff(n,d) ~|~F_j=F'_j \text{ for }i\neq j,~F_i\supset F'_i,~\dim(F_i/F'_i)=1,~\dim(F_j)=d_j \text{ for all }j=1,\ldots,n-1\}$, $f_i^{d_1,\ldots,d_n}$ is the characteristic function of the orbit $\{(F,F')\in\ff(n,d)\times\ff(n,d) ~|~F_j=F'_j \text{ for }i\neq j,~F_i\subset F'_i,~\dim(F'_i/F_i)=1,~\dim(F_j)=d_j \text{ for all }j=1,\ldots,n-1\}$. Additionally, the identity of the algebra is $\mathbf{1}=\sum\limits_{\substack{(d_1,\ldots,d_n)\in\ZZ_{\geq0}^n\\\sum_i d_i=d}}\mathbf{1}^{d_1,\ldots,d_n}$.

The quantum Schur-Weyl duality of the commuting actions of $\uvsln$ and $H_d$ on $V^{\otimes d}$ (with $V=\CC(v)^n$) then can be seen as coming from the convolution action of the two convolution algebras on either side of the space $V_{n,d}=\CC\left(\ff(n,d)\times\ff(d)\right)^{G}$ (see \cite{J86} and \cite{GL}).

As was shown in \cite{MWZ}, the diagonal action of $G$ on $\ff(d)\times\ff(d)\times \FF_q^d$ (resp. $\ff(n,d)\times\ff(n,d)\times \FF_q^d$) 
also has finitely many orbits (independently of $q$) that can be parametrized by \emph{decorated matrices}. The space of invariant functions $\mm\mathcal{H}_{q,d}:=\CC\left(\ff(d)\times\ff(d)\times \FF_q^d\right)^{G}$ (resp. $M\mathcal{U}_{q}(n,d):=\CC\left(\ff(n,d)\times\ff(n,d)\times \FF_q^d\right)^{G}$) is a unital associative $\CC$-algebra with the convolution product defined by
\begin{equation} \label{eq:convo}
(\alpha*\beta)(F,F',u):=\sum_{F''\in\ff(n,d),~u'\in\FF_q^d}\alpha(F,F'',u')\beta(F'',F',u-u'),
\end{equation}
(replace with $F''\in\ff(n,d)$ for $M\mathcal{U}_{q}(n,d)$) for any $\alpha,\beta\in \mm\mathcal{H}_{q,d}$ (resp. $M\mathcal{U}_{q}(n,d)$), where the structure constants are again polynomials in $q$. We can then view $\mm\mathcal{H}_{q,d}$ (resp. $M\mathcal{U}_{q}(n,d)$) to be the specialization of the mirabolic Hecke algebra $MH_d$ (resp. an algebra $MU(n,d)$) over $K=\CC(v)$, when $v=\sqrt{q}$. We call $MU(n,d)$ the mirabolic quantum Schur algebra.

As in the usual Hecke algebra, the identity in the algebra $MH_d$ is the characteristic function of the $G$-orbit $\{(F,F',0)\in\ff(d)\times\ff(d)\times \FF_q^d ~|~F_i=F'_i \text{ for all }i=1,\ldots,d\}$, and the generator $\tau_j$ for $j>0$ corresponds to the characteristic function of the $G$-orbit $\{(F,F',0)\in\ff(d)\times\ff(d)\times \FF_q^d ~|~F_i=F'_i \text{ for }i\neq j,~F_j\cap F'_j=F_{j-1}=F'_{j-1}\}$. Finally, the generator $\tau_0\in MH_d$ (in the notation of \cite{FZM}) is given by the characteristic function of the $G$-orbit $\{(F,F',v)\in\ff(d)\times\ff(d)\times \FF_q^d ~|~F_i=F'_i \text{ for all }i=1,\ldots,d, v\in F_1\sm\{0\}\}$ and $e=v^{-2}(\tau_0+1)$.

Analogously to what we saw above for the quantum Schur algebra, we can describe the generators of $MU(n,d)$ as follows: 
$$k_i^{\pm 1}=\sum_{\substack{(d_1,\ldots,d_n)\in\ZZ_{\geq0}^n\\\sum_i d_i=d}}v^{\pm(d_i-d_{i+1})}\mathbf{1}^{d_1,\ldots,d_n},\quad e_i=\sum_{\substack{(d_1,\ldots,d_n)\in\ZZ_{\geq0}^n\\\sum_i d_i=d}}v^{1-d_i}e_i^{d_1,\ldots,d_n}$$
$$f_i=\sum_{\substack{(d_1,\ldots,d_n)\in\ZZ_{\geq0}^n\\\sum_i d_i=d}}v^{1-d_{i+1}}f_i^{d_1,\ldots,d_n},\quad \ell=\sum_{\substack{(d_1,\ldots,d_n)\in\ZZ_{\geq0}^n\\\sum_i d_i=d}}v^{-2d_1}\ell^{d_1,\ldots,d_n},$$
where $\mathbf{1}^{d_1,\ldots,d_n}$ is the characteristic function of the orbit $\{(F,F',0)\in\ff(n,d)\times\ff(n,d)\times\FF_q^d ~|~F_j=F'_j,~\dim(F_j)=d_j \text{ for all }j=1,\ldots,n-1\}$, $e_i^{d_1,\ldots,d_n}$ is the characteristic function of the orbit $\{(F,F',0)\in\ff(n,d)\times\ff(n,d)\times\FF_q^d ~|~F_j=F'_j \text{ for }i\neq j,~F_i\supset F'_i,~\dim(F_i/F'_i)=1,~\dim(F_j)=d_j \text{ for all }j=1,\ldots,n-1\}$, $f_i^{d_1,\ldots,d_n}$ is the characteristic function of the orbit $\{(F,F',0)\in\ff(n,d)\times\ff(n,d)\times\FF_q^d ~|~F_j=F'_j \text{ for }i\neq j,~F_i\subset F'_i,~\dim(F'_i/F_i)=1,~\dim(F_j)=d_j \text{ for all }j=1,\ldots,n-1\}$, and $\ell^{d_1,\ldots,d_n}$ is the characteristic function of the $G$-invariant subvariety $\{(F,F',u)\in\ff(n,d)\times\ff(n,d)\times\FF_q^d ~|~F_j=F'_j,~\dim(F_j)=d_j \text{ for all }j=1,\ldots,n-1,~u\in F_1\}$. Again, the identity of the algebra is $\mathbf{1}=\sum\limits_{\substack{(d_1,\ldots,d_n)\in\ZZ_{\geq0}^n\\\sum_i d_i=d}}\mathbf{1}^{d_1,\ldots,d_n}$.

By \cite[Proposition 4.2.5]{FZM}, the generators of $MU(n,d)$ satisfy the relations in Definition \ref{def:uvsln} (in addition to other relations that depend on $d$) and hence we have a surjective map $MU(n)\onto MU(n,d)$ for all $d$. We notice here that in \cite{FZM} the authors prefer to work with the `general linear' analogue of the mirabolic quantum group, while we are using the `special linear' version, which is why the `Cartan part' has $n$ generators in their case ($H_a$, $1\leq a\leq n$, see \cite[Def. 2.1.2]{FZM}), and $n-1$ in our case ($k_i$, $1\leq i\leq n-1$). This difference is not significant for our purposes.

Now, we consider the vector space $\CC(\ff(n,d)\times\ff(d)\times \FF_q^d)^{G}$, which has a left action of $\mm\uu_q(n,d)$ and a right action of $M\mathcal{H}_{q,d}$ given by essentially the same formula as \eqref{eq:convo}, summing over the middle partial flag or complete flag. Then these actions centralize each other (see \cite[$\mathsection 6$]{R15}), giving us a mirabolic analogue of the geometric construction of the quantum Schur-Weyl duality. Again the structure constants are polynomials in $q$, so we can consider this to be the specialization of a $K$-vector space $MV_{n,d}$, which we call \emph{mirabolic tensor space} (despite it not actually being a tensor product), with a left action of $MU(n,d)$ and a right action of $MH_d$. We will describe this in more details in the next section.

\section{Mirabolic quantum Schur-Weyl duality}\label{sec:MSWD}

\subsection{Main theorem}

By the discussion on convolution at the end of the previous section, we have a finite dimensional $K$-vector space $MV_{n,d}$ with a left action of the mirabolic quantum Schur algebra $MU(n,d)$ and a right action of the mirabolic Hecke algebra $MH_d$ and the two actions centralize each other. More precisely, by \cite[Lemma 6.6]{R15}, we have the equality:
\[\End_{MH_d}(MV_{n,d})=MU(n,d).\]
As a consequence, since $MH_d$ is semisimple, by the double centralizer theorem (see for example \cite[Theorem 5.18.1]{Eting}), we have that $\End_{MU(n,d)}(MV_{n,d})$ is equal to the image of $MH_d$ in $\End(MV_{n,d})$. Furthermore, the mirabolic tensor space $MV_{n,d}$ decomposes as an $(MU(n,d),MH_d)$-bimodule:
\begin{equation} \label{eq:DCT}
MV_{n,d} = \bigoplus_{\bs\lambda\in \Lambda} L_{\bs\lambda}\otimes M^{\bs\lambda},
\end{equation}
where the $L_{\bs\lambda}$'s (resp. the $M^{\bs\lambda}$'s) are pairwise non-isomorphic simple modules over the algebra $MU(n,d)$ (resp. $MH_d$) and $\Lambda$ is some finite index set. Since $MU(n,d)$ is a quotient of the mirabolic quantum group $MU(n)$, we can view $MV_{n,d}$ as an $(MU(n),MH_d)$-bimodule. The main theorem of this paper describes the above decomposition explicitly in terms of the already classified $MU(n)$-modules and $MH_d$-modules:

\begin{theorem} \label{thm:MSWD}
As an $(MU(n),MH_d)$-bimodule, we have a decomposition of the mirabolic tensor space:
\begin{equation} \label{eq:MSWD}
MV_{n,d} = \bigoplus_{\bs\lambda\in M\Lambda_{n,d}} L_{\bs\lambda}\otimes M^{\bs\lambda}.
\end{equation}
Here, $M\Lambda_{n,d}=\{\bs\lambda=(\lambda,1^s): |\lambda|+s=d, s\leq n \text{ and } \lambda \text { has at most }n \text{ parts}\}$.
\end{theorem}

The theorem asserts that for any bipartition $\bs\lambda\in M\Lambda_{n,d}$, the isotypic components of the simple $MU(n)$-module $L_{\bs\lambda}$ and the simple $MH_d$-module $M^{\bs\lambda}$ in $MV_{n,d}$ are non-zero and coincide. We will prove this in the following sections by showing:
\begin{enumerate}
\item The isotypic component of $L_{\bs\lambda}$ in $MV_{n,d}$ is non-zero if and only if $\bs\lambda\in M\Lambda_{n,d}$ (see Section~\ref{sec:leftmod}).
\item The isotypic component of $L_{\bs\lambda}$ in $MV_{n,d}$ is equal to the isotypic component of $M^{\bs\lambda}$ (see Section~\ref{sec:MSWDproof}).
\end{enumerate}

\begin{remark} \label{rem:TL}
Since $MH_d$ is a semisimple algebra, the element $1\in MH_d$ can be expressed as a sum of commuting idempotents, each corresponding to the projection onto the isotypic component of a different irreducible representation of $MH_d$. In particular, there exists an idempotent $\mathbf{e}_d\in MH_d$ which corresponds to projecting onto the representation $M^{(\emptyset,1^d)}$. It can be verified that this idempotent is given by:
\[\mathbf{e}_d=\prod_{i=1}^d(1-\LL_i),\]
where $\LL_i$ is the Jucys-Murphy element. Similarly, there exists another idempotent $\mathbf{e}^d\in MH_d$ which corresponds to the projection onto the representation $M^{(1^d,\emptyset)}$, and is given by the formula:
\[\mathbf{e}^d=\frac{(-v)^{\frac{d(d-1)}{2}}}{[d]_v!}\prod_{i=1}^d\Bigg(\LL_i\prod_{j=1}^{i-1}\Big(\tau_{i-j}-\frac{v^{2(i-j+1)}-v^{2(i-j)}}{v^{2(i-j)}-1}\Big)\Bigg).\]

Here, we assume that the product is multiplied from left to right as the indices $i$ and $j$ increase, and $[d]_v!$ represents the $v$-factorial:
\[[d]_v!=\prod_{i=1}^d\frac{v^n-v^{-n}}{v-v^{-1}}.\]

Note that the elements $\mathbf{e}_d$ and $\mathbf{e}^d$ can be viewed as elements of $MH_k$ for any $k\geq d$. We can think of these as mirabolic analogs of the $d$-antisymmetrizer element in the Iwahori-Hecke algebra that annihilates the sign representation. As a corollary of the theorem above, the two-sided ideal generated by the elements $\mathbf{e}_{d+1}$ and $\mathbf{e}^{d+1}$ is exactly the kernel of the $MH_k$-action on the mirabolic tensor space $MV_{d,k}$ whenever $k\geq d+1$. In particular, the quotient of the algebra $MH_d$ by the two-sided ideal generated by $\mathbf{e}_3$ and $\mathbf{e}^3$ should give a mirabolic analog of the \emph{Temperley-Lieb algebra}.
\end{remark}

\subsection{Explicit formulae for action on the mirabolic tensor space} \label{sec:action-formulae}
In this section we explicitly describe a basis for $MV_{n,d}$ and formulae for the action of $MU(n)$ and $MH_d$ on it, following \cite{FZM}. Remember from \cite{R15} that $\GL_d$-orbits on triples of two flags (complete or partial) and a vector are parametrized by decorated matrices. We will only need the case of one partial flag, one complete flag, and a vector.
\begin{definition}
For fixed $n,d\geq 1$, define the set 
$$S_{row}:=\{\mathbf{r}=(r_1,\ldots,r_d)~|~1\leq r_j\leq n\}=\{1,2,\dots, n\}^d$$
and for each $\mathbf{r}\in S_{row}$, define $$J_{\mathbf{r}}:= \{\mathbf{j}\subset\{1,\ldots,d\}~|~\text{ if }j_1,j_2\in\mathbf{j} \text{ with }j_1<j_2, \text{ then }r_{j_1}>r_{j_2}\}.$$
We define the set 
$$\Xi_{n,1^d}:=\{(\mathbf{r},\mathbf{j})~|~\mathbf{r}\in S_{row},~\mathbf{j}\in J_{\mathbf{r}}\}.$$ 
\end{definition}

Any element $\mathbf{r}\in S_{row}$ can also be viewed as an $n\times d$ matrix such that for each $1\leq i\leq d$, the $r_i^{th}$ entry of the $i^{th}$ column is one, and all the other entries of the column are zero. Thus, $S_{row}$ is the set of $n\times d$ matrices with non-negative integer entries, such that each column sums to one. Similarly, the set $\Xi_{n,1^d}$ can be viewed as the set of decorated matrices: Given a pair $(\mathbf{r},\mathbf{j})\in\Xi_{n,1^d}$, the element $\mathbf{r}$ can be viewed as an $n\times d$ matrix as above and the element $\mathbf{j}=\{j_1<j_2<\cdots<j_k\}$ dictates that the non-zero entry of the $j_i^{th}$ column must be decorated for each $i$, with the decreasing condition on $r_j$ meaning that the decorated entries form a diagonal from bottom left to top right. In what proceeds, we will circle matrix entries to indicate decorations. 

\begin{example}
Fix $n=3$ and $d=5$. The element $\bfr=(3,2,3,1,2)\in S_{row}$ can be viewed as the matrix:
\[\begin{pmatrix}
0&0&0&1&0\\0&1&0&0&1\\1&0&1&0&0
\end{pmatrix}.\]
For the same $\bfr$ and $\bfj=(1,2,4)$, we can view the pair $(\bfr,\bfj)$ as the decorated matrix:
\[\begin{pmatrix}
0&0&0&\circled1&0\\0&\circled1&0&0&1\\\circled1&0&1&0&0
\end{pmatrix}.\]
\end{example}

\begin{remark} \label{rem:bfr-bijection}
By \cite{MWZ}, there is a bijection $G\backslash(\ff(n,d)\times\ff(d)\times \FF_q^d)\longleftrightarrow \Xi_{n,1^d}$ which can be described as follows: To an element $(F,F',u)\in\ff(n,d)\times\ff(d)\times \FF_q^d,$ we associate an $n\times d$ matrix $A(F,F',u)=(a_{i,j})$ whose entries are given by:
\[a_{i,j}=\dim\Big(\frac{F_i\cap F_j'}{F_{i-1}\cap F_j'+F_i\cap F_{j-1}'}\Big).\]
We also construct a subset $\Delta(F,F',u) \sub \{1,2, \dots, d\}$ which is the intersection of all possible subsets $\Delta'\sub \{1,2, \dots, d\}$ such that:
\[u\in \sum\limits_{j\in \Delta'}\sum\limits_{i:a_{i,j}\neq 0} F_i\cap F_j'.\]
Then, the above bijection is given by:
\[(F,F',u)\mapsto (A(F,F',u), \Delta(F,F',u)).\]
(Here, we have viewed an element of $\Xi_{n,1^d}$ as an $n\times d$ matrix $A(F,F',u)$ with decorations in the columns indexed by $\Delta(F,F',u)$.)
Hence we can consider $MV_{n,d}$ to be a $K$-vector space, with basis  
$$\{m_{\mathbf{r},\mathbf{j}}~|~(\mathbf{r},\mathbf{j})\in \Xi_{n,1^d}\}.$$
By abuse of notation, we will often identify the basis element $m_{\bfr, \bfj}$ with the corresponding $n\times d$ matrix.
\end{remark}

We now 
collect all the formulae for the action of the algebras $MH_d$ and $MU(n)$ on the mirabolic tensor space that we need for our proofs.

We start with the right action of $MH_d$ on $MV_{n,d}$. The formulae from \cite[Lemma 2.2.2]{FZM} give us
$$m_{\mathbf{r},\mathbf{j}}\cdot\tau_0=\begin{cases}(v^2-1)m_{\mathbf{r},\mathbf{j}} & \text{ if }1\not\in\mathbf{j},~\exists~ j\in\mathbf{j} \text{ s.t. }r_1\leq r_j,\\ m_{\mathbf{r},\mathbf{j}\cup\{1\}} & \text{ if }1\not\in\mathbf{j},~\forall~ j\in\mathbf{j}~ r_1> r_j,\\
(v^2-2)m_{\mathbf{r},\mathbf{j}}+(v^2-1)m_{\mathbf{r},\mathbf{j}\setminus\{1\}} & \text{ if }1\in\mathbf{j}.\end{cases}$$
So, since $e=v^{-2}(\tau_0+1)$, it follows that
$$m_{\mathbf{r},\mathbf{j}}\cdot e=\begin{cases}m_{\mathbf{r},\mathbf{j}} & \text{ if }1\not\in\mathbf{j},~\exists~ j\in\mathbf{j} \text{ s.t. }r_1\leq r_j,\\ v^{-2}m_{\mathbf{r},\mathbf{j}}+v^{-2}m_{\mathbf{r},\mathbf{j}\cup\{1\}} & \text{ if }1\not\in\mathbf{j},~\forall~ j\in\mathbf{j}~ r_1> r_j,\\
(1-v^{-2})m_{\mathbf{r},\mathbf{j}}+(1-v^{-2})m_{\mathbf{r},\mathbf{j}\setminus\{1\}} & \text{ if }1\in\mathbf{j}.\end{cases}$$

\begin{corollary} \label{cor:eeigen}
The eigenvalues of $e$ are $0$ and $1$. The $0$-eigenspace i.e. the kernel for the $e$-action has a basis
$$\{m_{\mathbf{r},\mathbf{j}}-(v^2-1)m_{\mathbf{r},\mathbf{j}\setminus\{1\}}~|~1\in\mathbf{j}\}$$
while the $1$-eigenspace (i.e. $\im(e)$) has a basis
$$\{m_{\mathbf{r},\mathbf{j}}+m_{\mathbf{r},\mathbf{j}\setminus\{1\}}~|~1\in\mathbf{j}\}\cup\{m_{\mathbf{r},\mathbf{j}}~|~1\not\in\mathbf{j},~\exists~ j\in\mathbf{j} \text{ s.t. }r_1\leq r_j \}.$$

In particular
$$\dim(\ker(e))=\#\{(\mathbf{r},\mathbf{j})\in \Xi_{n,1^d}~|~1\in\mathbf{j}\} $$
and
$$\dim(\im(e))=\#\{(\mathbf{r},\mathbf{j})\in \Xi_{n,1^d}~|~1\not\in\mathbf{j}\}.$$
\end{corollary}

Next, we recall the action of $\tau_j$ for $j>0$ from \cite[Lemma 2.2.2]{FZM}. Let $\mathcal{S}_d$ denote the symmetric group on $d$ letters and let $s_j\in \mathcal{S}_d$ denote the transposition $(j,j+1)$. The group $\mathcal{S}_d$ acts on $S_{row}$ by $s_j(r_1,\ldots,r_j,r_{j+1},\ldots,r_d)=(r_1,\ldots,r_{j+1},r_j,\ldots,r_d)$ and on $J_{\mathbf{r}}$ by its defining action as permutations of $\{1,\ldots,d\}$. 

If $j,j+1\not\in\mathbf{j}$, then

$$m_{\mathbf{r},\mathbf{j}}\cdot\tau_j=\begin{cases}m_{s_j\mathbf{r},s_j\mathbf{j}} & \text{ if }r_j< r_{j+1},\\
v^2 m_{s_j\mathbf{r},s_j\mathbf{j}} & \text{ if }r_j= r_{j+1},\\
v^2m_{s_j\mathbf{r},s_j\mathbf{j}}+(v^2-1)m_{\mathbf{r},\mathbf{j}} & \text{ if }r_j> r_{j+1}.\end{cases}$$

If $j\in \mathbf{j}$, $j+1\not\in \mathbf{j}$, then
$$m_{\mathbf{r},\mathbf{j}}\cdot\tau_j=\begin{cases}m_{s_j\mathbf{r},s_j\mathbf{j}} & \text{ if }r_j\leq r_{j+1},\\
m_{s_j\mathbf{r},s_j\mathbf{j}}+m_{\mathbf{r},\mathbf{j}\cup\{j+1\}} & \text{ if }r_j> r_{j+1},~\forall k\in\mathbf{j},~k>j+1 \Rightarrow r_{j+1}> r_k,\\
v^2m_{s_j\mathbf{r},s_j\mathbf{j}}+(v^2-1)m_{\mathbf{r},\mathbf{j}} & \text{ if }r_j> r_{j+1},~\exists k\in\mathbf{j} \text{ s.t. }k>j+1,r_{j+1}\leq r_k.\end{cases}$$

If $j\not\in \mathbf{j}$, $j+1\in \mathbf{j}$, then
$$m_{\mathbf{r},\mathbf{j}}\cdot\tau_j=\begin{cases}v^2m_{s_j\mathbf{r},s_j\mathbf{j}}+(v^2-1)m_{\mathbf{r},\mathbf{j}} & \text{ if }r_j\geq r_{j+1},\\
m_{s_j\mathbf{r},s_j\mathbf{j}}+m_{s_j\mathbf{r},\mathbf{j}\cup\{j\}} & \text{ if }r_j< r_{j+1},~\forall k\in\mathbf{j},~k>j+1 \Rightarrow r_{j}> r_k,\\
m_{s_j\mathbf{r},s_j\mathbf{j}} & \text{ if }r_j< r_{j+1},~\exists k\in\mathbf{j} \text{ s.t. }k>j+1,r_{j}\leq r_k.\end{cases}$$

If $j,j+1\in\mathbf{j}$, then necessarily $r_j>r_{j+1}$ and we have
$$m_{\mathbf{r},\mathbf{j}}\cdot\tau_j=
(v^2-1)m_{s_j\mathbf{r},\mathbf{j}\setminus\{j\}}+(v^2-1)m_{\mathbf{r},\mathbf{j}\setminus\{j+1\}}+(v^2-2)m_{\mathbf{r},\mathbf{j}}.$$

For the left action of $MU(n)$ on $MV_{n,d}$, we will only need a formula for the action of $\ell, k^{\pm}_i\in MU(n)$ from \cite[Lemma 2.2.1]{FZM}:

\begin{equation}\label{eq:kaction}k_{i}^{\pm 1}\cdot m_{\bfr,\bfj} = v^{\pm(\#\{j:r_j=i\}-\#\{j:r_j=i+1\})}m_{\bfr,\bfj}.\end{equation}

\begin{multline}\label{eq:laction}
\ell\cdot m_{\bfr,\bfj}= \\ \begin{cases}v^{-2\#\{i> \max\bfj~|~r_i=1\}}\left(m_{\bfr,\bfj}+\sum\limits_{k>\max\bfj,~r_k=1}m_{\bfr,\bfj\cup\{k\}} \right)  & \text{ if }\forall j\in\bfj, ~r_j\neq 1,\\
v^{-2\#\{i\geq \max\bfj~|~r_i=1\}}(v^2-1)\left(m_{\bfr,\bfj\setminus\{j\}}+\sum\limits_{k>\max(\bfj\setminus \{j\}),~r_k=1}m_{\bfr,(\bfj\setminus\{j\})\cup\{k\}} \right)  & \text{ if }\exists j\in\bfj,~r_j=1.
\end{cases}
\end{multline}
(We follow the convention that if $\emptyset=\bfj\subset\{1,\ldots,d\}$, then $\max\bfj:=0$.)

\begin{remark}
In the notation of \cite{FZM}, our Cartan generators $k_{i}^{\pm1}$ are equal to $(H_iH_{i+1}^{-1})^{\pm1}$.
\end{remark}

\begin{corollary} \label{cor:lconse}
Consider any $m_{\mathbf{r},\mathbf{j}}$ such that $j\in\mathbf{j}$ and $r_j=1$. Then, we have:
\[\ell\cdot m_{\mathbf{r},\mathbf{j}}=v^{2\#\{\max(\bfj\sm \{j\})<i<j~|~r_i=1\}}(v^2-1)\ell\cdot m_{\mathbf{r},\mathbf{j}\sm\{j\}}.\]
In particular, if $r_j=1$ for some $j\in \bfj$, then $\ell\cdot m_{\bfr,\bfj}$ is a non-zero scalar multiple of $\ell\cdot m_{\bfr,\bfj\sm\{j\}}$. As a consequence, $\mathrm{Im}(\ell)$ has a basis given by:
\[\{\ell\cdot m_{\bfr,\bfj}: \forall j\in\bfj, ~r_j\neq 1\}.\]
 Furthermore, if $j$ is minimal amongst all $i$'s such that $r_i=1$, then:
\[\ell\cdot m_{\mathbf{r},\mathbf{j}}=(v^2-1)\ell\cdot m_{\mathbf{r},\mathbf{j}\sm\{j\}}.\]
\end{corollary}
\subsection{The case $s=0$: quantum Schur-Weyl duality} \label{sec:MSWDs=0}

In this section, we prove the part of Theorem~\ref{thm:MSWD}, when $s=0$ (that is, for those bipartitions $\bs\lambda$' s for which the second partition is empty) by showing that in this case, the correspondence is equivalent to the quantum Schur-Weyl duality of Jimbo (see \cite{J86}) which we now recall.

Let $V_{n,d}:=(K^n)^{\otimes d}$ denote the quantum tensor space, which has a natural left action of the quantum Schur algebra $U(n,d)$ and right action of the Iwahori-Hecke algebra $H_d$. The algebra $U(n,d)$ is a quotient of the quantum group $\uvsln$ and its left action on $V_{n,d}$ lifts to an $\uvsln$-action. Then, we have the following duality statement:

\begin{theorem} \label{thm:QSWD}
The respective actions of the algebras $\uvsln$ and $H_d$ on the quantum tensor space $V_{n,d}$ commute with each other. Furthermore, the space $V_{n,d}$ has the following decomposition as a $(\uvsln,H_d)$-bimodule:
\[V_{n,d} = \bigoplus_{\lambda\in\Lambda_{n,d}} L_{\lambda}\otimes M^{\lambda},\]
where $\Lambda_{n,d}$ is the set of all partitions $\lambda$ of $d$ with at most $n$ parts, $L_\lambda$ is the simple $\uvsln$-module with highest weight $\lambda$, and $M^{\lambda}$ is the simple $H_d$-module corresponding to the partition $\lambda$.
\end{theorem}

As explained in Section \S 2.3, this duality can be seen as coming from the convolution action on functions over the variety of pairs of one partial and one complete flag (also see \cite{GL}). 
 
We consider the map defined in terms of convolution algebras by:
\[\pi_1:\CC(\ff(n,d)\times\ff(n,d)\times \FF_q^d)^{G}\ra\CC(\ff(n,d)\times\ff(n,d))^{G} ,\]
where, for any $\alpha\in\CC(\ff(n,d)\times\ff(n,d)\times \FF_q^d)^{G},$ we have that:
\begin{equation}\label{eq:conv-surj}\pi_1(\alpha)(F,F') = \sum_{u\in\FF_q^d}\alpha(F,F',u).\end{equation}
\begin{lemma}\label{lem:conv-pi1}The map $\pi_1$ gives rise to a surjection of the generic algebras $MU(n,d) \twoheadrightarrow U(n,d)$, which is compatible with the map $\pi_1:MU(n)\to \uvsln$ of Remark \ref{rem:pi0-pi1}. That is, we have the following commutative diagram:
\[\begin{tikzcd}
	  {MU(n)} & {\uvsln} \\
	 {MU(n,d)} & {U(n,d)}
	\arrow["\pi_1", two heads, from=1-1, to=1-2]
	\arrow[two heads, from=1-1, to=2-1]
	\arrow[two heads, from=1-2, to=2-2]
	\arrow["\pi_1", two heads, from=2-1, to=2-2]
\end{tikzcd}.\]
\end{lemma}
\begin{proof}
It is straightforward to verify that $\pi_1(\alpha*\beta)=\pi_1(\alpha)*\pi_1(\beta)$ for all $\alpha,\beta\in \CC(\ff(n,d)\times\ff(n,d)\times \FF_q^d)^{G}$, hence $\pi_1$ induces an algebra map of the specialized convolution algebras. 
Given the definition of the generators in terms of invariant functions on flags, as explained in Section \ref{sec:conv}, we have that $\pi_1$ sends the generators $e_i,f_i,k_i^{\pm1}$ to the elements of the same name in $\uvsln$, hence proving surjectivity and the compatibility with Remark \ref{rem:pi0-pi1}. Finally, we need to check that $\pi_1(\ell)=1$, for the generator $\ell\in\mm\uu_q(n,d)$, which can be written (as in Section \ref{sec:conv})  as
$$\ell=\sum_{\substack{(d_1,\ldots,d_n)\in\ZZ_{\geq0}^n\\\sum_i d_i=d}}v^{-2d_1}\ell^{d_1,\ldots,d_n}$$
where
$\ell^{d_1,\ldots,d_n}$  the characteristic function of the $G$-invariant subvariety $\{(F,F',u)\in\ff(n,d)\times\ff(n,d)\times\FF_q^d ~|~F_j=F'_j,~\dim(F_j)=d_j \text{ for all }j=1,\ldots,n-1,~u\in F_1\}$.
We then have 
\begin{align*}\pi_1(\ell^{d_1,\ldots,d_n})(F,F')&=\sum_{u\in \FF_q^d}\ell^{d_1,\ldots,d_n}(F,F',u) \\
&=\begin{cases} v^{2\dim(F_1)} & \text{ if }F_i=F'_i,~\dim(F_i)=d_i,~\forall ~i \\ 0 & \text{ otherwise. }\end{cases}\end{align*}
Hence, 
\begin{align*}\pi_1(\ell)(F,F')&=\sum_{\substack{(d_1,\ldots,d_n)\in\ZZ_{\geq0}^n\\\sum_i d_i=d}}v^{-2d_1}\pi_1(\ell^{d_1,\ldots,d_n})(F,F')\\
&= \begin{cases} v^{-2\dim(F_1)}\pi_1(\ell^{\dim(F_1),\ldots,\dim(F_n)})(F,F') & \text{ if }F_i=F'_i~\forall~i \\ 0 & \text{ else }\end{cases}\\
&= \begin{cases} v^{-2\dim(F_1)}v^{2\dim(F_1)} =1 & \text{ if }F_i=F'_i~\forall~i \\ 0 & \text{ otherwise, }\end{cases}\end{align*}
which proves indeed that $\pi_1(\ell)$ is the identity for the convolution product in $\CC(\ff(n,d)\times\ff(n,d))^{G}$.
The result then follows from considering the generic convolution algebras.
\end{proof}
Analogously to Lemma \ref{lem:conv-pi1} (and using the same arguments in its proof to check that $p(e)=1$), we get that the projection $p:MH_d\to H_d$ of Remark \ref{rem:p-proj}(1) is given, on the level of their $v=\sqrt{q}$ specializations, via:
\[p:\CC(\ff(d)\times \ff(d)\times \FF_q^d)^{G}\to\CC(\ff(d)\times \ff(d))^{G},\]
where $p$ is defined by same formula as in \eqref{eq:conv-surj}.

Finally, we construct the pullback map from the specialization of the usual quantum tensor space $V_{n,d}$ to that of the mirabolic tensor space $MV_{n,d}$:
\[\phi:\CC(\ff(n,d)\times \ff(d))^{G}\ra \CC(\ff(n,d)\times \ff(d)\times \FF_q^d)^{G}\]
such that for any $\gamma\in \CC(\ff(n,d)\times \ff(d))^{G}$, we have:
\[\phi(\gamma)(F,F',u):=\gamma(F,F').\]
By construction $\phi$ is injective, hence induces an embedding of the generic $K$-vector spaces $\phi:V_{n,d}\hookrightarrow MV_{n,d}.$

Combinatorially, the space $V_{n,d}$ has a basis given by the set $\{m_{\mathbf{r}}:\mathbf{r}\in S_{row}\}$, and each of these basis vectors can be viewed as an $n\times d$ matrix. In terms of this basis, the map $\phi$ is given by:
\[\phi(m_{\mathbf{r}})=\sum_{\mathbf{j}\in J_{\mathbf{r}}}m_{\mathbf{r},\mathbf{j}}.\]
That is, the map $\phi$ sends an $n\times d$ matrix to a sum over all possible valid decorations of that matrix.
\begin{lemma}\label{lem:phi-pullback}
Let $\alpha\in\CC(\ff(n,d)\times \ff(n,d)\times \FF_q^d)^{G}$, $\gamma\in\CC(\ff(n,d)\times \ff(d))^{G}$, and $\beta\in\CC(\ff(d)\times \ff(d)\times \FF_q^d)^{G}$, then
$$\phi(\pi_1(\alpha)*\gamma)=\alpha*\phi(\gamma),\quad \phi(\gamma*p(\beta))=\phi(\gamma)*\beta.$$
\end{lemma}
\begin{proof}
We prove the first equality, the second one is entirely analogous.
\begin{align*}
\phi(\pi_1(\alpha)*\gamma)(F,F',u)&=(\pi_1(\alpha)*\gamma)(F,F') \\
&= \sum_{F''\in \ff(n,d)}\pi_1(\alpha)(F,F'')\gamma(F'',F') \\
&= \sum_{F''\in\ff(n,d),~u'\in\FF_q^d}\alpha(F,F'',u')\gamma(F'',F') \\
&= \sum_{F''\in\ff(n,d),~u'\in\FF_q^d}\alpha(F,F'',u')\phi(\gamma)(F'',F',u-u') \\
&= (\alpha*\phi(\gamma))(F,F',u).
\end{align*}
\end{proof}
\begin{proposition} \label{prop:s=0}
Let $\lambda$ be any partition of $d$ with at most $n$ parts and let $\bs\lambda=(\lambda,\emptyset)$. Then, the isotypic component of the $MU(n)$-module $L_{\bs\lambda}$ in $MV_{n,d}$ is non-zero and is equal to the isotypic component of the $MH_d$-module $M^{\bs\lambda}$.
\end{proposition}
\begin{proof}
We consider the embedding $\phi:V_{n,d}\hookrightarrow MV_{n,d}$. By Lemma \ref{lem:phi-pullback}, we have that $\im\phi\subset MV_{n,d}$ is invariant under the actions of $MU(n)$ on the left and of $MH_d$ on the right. Further, again by Lemma \ref{lem:phi-pullback}, the action of $MU(n)$ (resp. $MH_d$) on $\im\phi$ factors through the action of $\uvsln$ (resp. $H_d$) on $V_{n,d}$, via the projection $\pi_1$ (resp. $p$). So, from Theorem \ref{thm:QSWD}, together with Remarks \ref{rem:pi0-pi1} and \ref{rem:p-proj}(1), we get the isomorphism of $(MU(n),MH_d)$-bimodules
$$\im\phi\simeq \bigoplus_{\lambda\in\Lambda_{n,d}} L_{(\lambda,\emptyset)}\otimes M^{(\lambda,\emptyset)}.$$
This isomorphism shows that for every partition $\lambda\in \Lambda_{n,d}$, the space $L_{(\lambda,\emptyset)}\otimes M^{(\lambda,\emptyset)}$ is a submodule of $MV_{n,d}$ as an $(MU(n),MH_d)$-bimodule. Therefore, as a consequence of~\eqref{eq:DCT}, it must be the isotypic component of $L_{(\lambda,\emptyset)}$ in $MV_{n,d}$ with respect to the $MU(n)$-action, completing the proof of the proposition.
\end{proof}

Fix $0\leq k\leq n$ and $0\leq t\leq d$, and define the subspace $MV_{n,d}^{k,t}\subseteq MV_{n,d}$, whose specialization at $v=\sqrt{q}$ is given by
\[MV_{n,d}^{k,t}=\{\gamma\in\CC(\ff(n,d)\times \ff(d) \times \FF_q^d)^{G}~|~\gamma(F,F',u)=0 ~\text{if}~F_k\cap F_t'\neq 0~\text{or}~ u\not\in F_t'\}.\]
In terms of $n\times d$ matrices, the space $MV_{n,d}^{k,t}$ is spanned by $m_{\bfr,\bfj}$ for all decorated matrices whose top-left $k\times t$ submatrix is zero and the only decorations are in the first $t$ columns. Note that $MV_{n,d}^{0,d}=MV_{n,d}$, whereas $MV_{n,d}^{n,0}=i(V_{n,d})$, where $i:V_{n,d}\hookrightarrow MV_{n,d}$ is the embedding given at the $v=\sqrt{q}$ specialization via:
\[i:\CC(\ff(n,d)\times \CC(\ff(d))^{G}\ra \CC(\ff(n,d)\times \CC(\ff(d)\times \FF_q^d)^{G},\]
where, for any $\alpha\in \CC(\ff(n,d)\times \CC(\ff(d))^{G},$ we define:
\begin{equation}\label{eq:conv-embed}i(\alpha)(F,F',u) = \begin{cases} \alpha(F,F') & \text{if } u=0,\\
0 & \text{if } u\neq0.
\end{cases}\end{equation}
Note that the embedding $i$ is not the same as $\phi$.


Define a map $\phi_{n,d}^{k,t}:MV_{n,d}^{k,t} \to MV_{n,d}$ which is given at the level of the $v=\sqrt{q}$ specialization as follows. For any $\gamma\in MV_{n,d}^{k,t}$, we define:
\begin{equation}\label{eq:phindkt}
\phi_{n,d}^{k,t}(\gamma)(F,F',u):=\begin{cases}\gamma(F,F',u'')& \text{if }F_k\cap F_t'=0~\text{and}~u\in F_k\oplus F_t'\\0 &\text{otherwise},\end{cases}
\end{equation}
where $u=(u',u'')\in F_k\oplus F_t'$. Note that the map $\phi_{n,d}^{0,d}$ is the identity map and $\phi=\phi^{n,0}_{n,d}\circ i$.

In terms of $n\times d$ matrices, the map $\phi_{n,d}^{k,t}$ can be described as follows. For any subset $S\sub \{1,2,\dots,n\}$, we define the set:
\begin{equation}\label{eq:jrs-def}
J_{\bfr}(S)=\{\bfj\in J_{\bfr}~|~\forall j\in\bfj~\text{and}~s\in S,~r_j< s\}.\end{equation}
For ease of notation, we will use $J_{\bfr}(k)$ to denote the set $J_{\bfr}(S)$ when $S=\{k\}$.
That is, $J_{\bfr}(k)$ is the set of all the possible decorations of the matrix $m_\bfr$ such that the decorations 
potentially exist only in the top $k-1$ rows. Then, the map $\phi_{n,d}^{k,t}$ is given by:
\begin{equation}\label{eq:phi-kt}
\phi_{n,d}^{k,t}: MV_{n,d}^{k,t}\to MV_{n,d},\qquad m_{\bfr,\bfj}\mapsto \sum_{\widetilde\bfj\in J_{\bfr}(k+1)} m_{\bfr,\bfj\cup\widetilde\bfj}.
\end{equation}
Note that the only decorations in $m_{\bfr,\bfj}$ are in the first $t$ columns and the bottom $n-k$ rows (since the top-left $k\times t$ submatrix is zero). Therefore, $\bfj\cup \widetilde\bfj$ is a valid decoration.

\begin{lemma} \label{lem:averows}
The space $MV_{n,d}^{k,t}$ is invariant under the left action of $e_i\in MU(n)$ when $i\neq k$ and the right action of $\tau_j\in MH_d$ when $1\leq j\neq t$. Furthermore, the map $\phi_{n,d}^{k,t}$ is equivariant with respect to these left and right actions.
\end{lemma}
\begin{proof}
This result can be proved using a basis of $MV_{n,d}^{k,t}$, together with \eqref{eq:phi-kt} and the formulae for the action on $MV_{n,d}$ from \cite{FZM}. However we offer a more conceptual proof using convolutions. First of all, we remark that, using the formulae for the generators in Section \ref{sec:conv}, the element $e_i$ in the convolution algebra is 
such that $e_i(F,F',u)=0$ unless $F_j=F_j'$ for all $j\neq i$, and $u=0$. Similarly, we have that $\tau_j\in MH_d$ 
satisfies $\tau_j(F,F',u)=0$ unless $F_i=F_i'$ for all $j\neq i$, and $u=0$.

Suppose $\gamma\in MV_{n,d}^{k,t}$, and $i\neq k$, then
\begin{align*}(e_i*\gamma)(F,F',u)&=\sum\limits_{F''\in \ff(n,d),~u'\in \FF_q^d}e_i(F,F'',u')\gamma(F'',F',u-u')\\
&= \sum\limits_{F''\in \ff(n,d)}e_i(F,F'',0)\gamma(F'',F',u) \\
&= \begin{cases} \sum\limits_{F''\in \ff(n,d)}e_i(F,F'',0)\gamma(F'',F',u) & \text{ if }F''_k\cap F'_t=0~\text{ and }u\in F'_t,\\
 0 & \text{ otherwise }\end{cases} \\
&= \begin{cases} \sum\limits_{F''\in \ff(n,d)}e_i(F,F'',0)\gamma(F'',F',u) & \text{ if }F_k\cap F'_t=0~\text{ and }u\in F'_t,\\
 0 & \text{ otherwise,}\end{cases}\end{align*}
showing that $e_i*\gamma\in MV_{n,d}^{k,t}$. Similarly, if $j\neq t$ then
\begin{align*}(\gamma*\tau_j)(F,F',u)&=\sum\limits_{F''\in \ff(d),~u'\in \FF_q^d}\gamma(F,F'',u')\tau_j(F'',F',u-u')\\
&= \sum\limits_{F''\in \ff(d)}\gamma(F,F'',u)\tau_j(F'',F',0) \\
&= \begin{cases} \sum\limits_{F''\in \ff(d)}\gamma(F,F'',u)\tau_j(F'',F',0) & \text{ if }F_k\cap F''_t=0~\text{ and }u\in F''_t,\\
 0 & \text{ otherwise }\end{cases} \\
&= \begin{cases} \sum\limits_{F''\in \ff(d)}\gamma(F,F'',u)\tau_j(F'',F',0) & \text{ if }F_k\cap F'_t=0~\text{ and }u\in F'_t,\\
 0 & \text{ otherwise,}\end{cases}\end{align*}
showing that $\gamma*\tau_j\in MV_{n,d}^{k,t}$. We also have
\begin{align*}(e_i*\phi_{n,d}^{k,t}(\gamma))(F,F',u)&=\sum\limits_{F''\in \ff(n,d),~u'\in \FF_q^d}e_i(F,F'',u')\phi_{n,d}^{k,t}(\gamma)(F'',F',u-u')\\
&=\sum\limits_{F''\in \ff(n,d)}e_i(F,F'',0)\phi_{n,d}^{k,t}(\gamma)(F'',F',u)\\
&= \begin{cases} \sum\limits_{F''\in \ff(n,d)}e_i(F,F'',0)\gamma(F'',F',u'') & \text{ if }F''_k\cap F'_t=0~\text{ and }u\in F''_k\oplus F'_t,\\
 0 & \text{ otherwise }\end{cases}\\
 &(\text{where } u=(u',u'')\in F''_k\oplus F'_t)\\
 &= \begin{cases} \sum\limits_{F''\in \ff(n,d)}e_i(F,F'',0)\gamma(F'',F',u'') & \text{ if }F_k\cap F'_t=0~\text{ and }u\in F_k\oplus F'_t,\\
 0 & \text{ otherwise }\end{cases}\\
 &= \begin{cases} (e_i*\gamma)(F,F',u'') & \text{ if }F_k\cap F'_t=0~\text{ and }u\in F_k\oplus F'_t,\\
 0 & \text{ otherwise }\end{cases}\\
&= \phi_{n,d}^{k,t}(e_i*\gamma)(F,F',u)
 \end{align*}
which shows that $\phi_{n,d}^{k,t}$ intertwines the action of $e_i$ for $i\neq k$. The computation to show that $\phi_{n,d}^{k,t}(\gamma*\tau_j)=(\phi_{n,d}^{k,t}(\gamma))*\tau_j$ for $j\neq t$ is completely analogous and will be omitted.
\end{proof}

\begin{corollary} \label{lem:averowssl}
Let $1\leq k\leq n-1$, the elements $1-\ell, (1-\ell)e_1, (1-\ell)e_1e_2, \dots, (1-\ell)e_1e_2\dots e_{k-1}$ of the mirabolic quantum group $MU(n)$ act by zero on $\textrm{Im}(\phi_{n,d}^{k,t})\sub MV_{n,d}$.
Consequently, $\textrm{Im}(\phi_{n,d}^{k,t})$ lies in the span of the isotypic components of those $L_{\bs\lambda}$'s that have co-depth at least $k$ (or depth at most $n-k$ by Lemma~\ref{lem:oldepth}).
\end{corollary}
\begin{proof}

Recall that $\ell=\sum_{d_1,\ldots,d_n}v^{-2d_1}\ell^{d_1,\ldots,d_n}$ as discussed in Section \ref{sec:conv}. 
Suppose that $\gamma\in MV^{k,t}_{n,d}$, 
so that $\gamma(F,F',u)= 0$ unless $F_k\cap F_t'= 0$ and $u\in F_t'$.
\begin{align*}
&(\ell*\phi^{k,t}_{n,d}(\gamma))(F,F',u)\\
&=\sum_{F''\in \ff(n,d),~u'\in\FF_q^d}\ell(F,F'',u')\phi^{k,t}_{n,d}(\gamma)(F'',F',u-u')\\
&= \sum_{d_1,\ldots,d_n}v^{-2d_1}\sum_{u'\in F_1}\ell^{d_1,\ldots,d_n}(F,F,u')\phi^{k,t}_{n,d}(\gamma)(F,F',u-u')\\
& \text{(since $\ell^{d_1,\ldots,d_n}(F,F'',u')\neq 0$ only if $F''=F$ and $u'\in F_1$) },\\
&=\begin{cases}\sum_{d_1,\ldots,d_n}v^{-2d_1}\sum_{u'\in F_1}\ell^{d_1,\ldots,d_n}(F,F,u')\gamma(F,F',u_2) &\text{if $F_k\cap F'_t=0$ and $u-u'\in F_k\oplus F'_t$} \\
0 & \text{otherwise}\end{cases}\\
&(\text{where } u-u'=(u_1,u_2)\in F_k\oplus F'_t)\\
&=\begin{cases} v^{-2\dim(F_1)}v^{2\dim(F_1)}\gamma(F,F',u_2) & \text{if $F_k\cap F'_t=0$ and $u\in F_k\oplus F'_t$},\\
0 & \text{otherwise} \end{cases}\\
&= \phi_{n,d}^{k,t}(\gamma)(F,F',u)\\
& (\text{since }u=(u_1+u',u_2)\in F_k\oplus F'_t).
\end{align*}
This shows that $\ell$ acts as the identity on $\im(\phi_{n,d}^{k,t})$, and hence $1-\ell$ acts as zero. The rest follows from Lemma \ref{lem:averows}, since $\im(\phi_{n,d}^{k,t})$ is invariant under the action of $e_i$ for $i<k$.
\end{proof}

\subsection{Proof of Mirabolic quantum Schur-Weyl duality as left modules} \label{sec:leftmod}

Recall the weight lattice $\ZZ^n/\sim$ of $\uvsln$ and consider the formal group ring $R:=\ZZ\langle v^{\mu} : \mu\in \ZZ^n/\sim\rangle$. For any finite dimensional $\uvsln$-module $M$, we define its formal character $\chi(M)$:
\[\chi(M):= \sum_{\mu}\dim(M_{\mu})v^{\mu}\in R,\]
where $M_{\mu}$ denotes the weight space of $M$ with weight $\mu.$ (Note that for any $i$, we have $\mathrm{Tr}(k_i|_M) = \sum_{\mu}\dim(M_{\mu})v^{\mu_i-\mu_{i+1}}=\sum_{\mu}\dim(M_{\mu})v^{(\mu,\alpha_i)}$.) By \cite[$\mathsection$ 5A.8]{J}, for any two finite dimensional $\uvsln$-modules $M$ and $N$, we have the equality:
\[\chi(M\otimes N)=\chi(M)\chi(N).\]

In this section, we prove the following proposition, which is the first step towards the proof of Theorem~\ref{thm:MSWD} and is a generalization of \cite[Theorem 6.9]{R15}:
\begin{proposition} \label{prop:leftmod}
As a left $MU(n)$-module, we have a decomposition of the mirabolic tensor space:
\begin{equation} \label{eq:leftmod}
MV_{n,d} = \bigoplus_{\bs\lambda\in M\Lambda_{n,d}} L_{\bs\lambda}\otimes M^{\bs\lambda}.
\end{equation}
Here, $M\Lambda_{n,d}=\{\bs\lambda=(\lambda,1^s): |\lambda|+s=d, s\leq n \text{ and } \lambda \text { has at most } n \text{ parts}\}$.
\end{proposition}

\begin{proof}
By Proposition~\ref{prop:weightspace}, it suffices to verify that the two sides of Equation~\eqref{eq:leftmod} have the same mirabolic weight spaces. 

In order to show that the two sides of~\eqref{eq:leftmod} have the same $k_i$-weight spaces, it suffices to show that their characters as $\uvsln$-modules are equal. To this end, we recall the decomposition from the usual quantum Schur-Weyl duality (see Theorem~\ref{thm:QSWD}):
\[V_{n,d} = \bigoplus_{\lambda\in\Lambda_{n,d}} L_{\lambda}\otimes M^{\lambda},\]
where $\Lambda_{n,d}$ is the set of all partitions $\lambda$ of $d$ with at most $n$ parts. Equating the characters $\chi$ of both sides, we have the equality:
\begin{equation} \label{eq:chillam}
\chi(V_{n,d})=\sum_{\lambda\in \Lambda_{n,d}}\chi(L_{\lambda})f^{\lambda},
\end{equation}
where $f^{\lambda}=\dim(M^{\lambda})$. Since the space $V_{n,d}$ is spanned by $n\times d$ matrices $m_{\bfr,\emptyset}$ for $\bfr\in S_{row}$, we can compute the character on the left explicitly. By~\eqref{eq:kaction}, the matrix $m_{\bfr,\emptyset}$ has weight $\mu=(\mu_1,\mu_2,\dots,\mu_n)\in \ZZ^n/\sim$ with respect to the $k_i$'s, where $\mu_1+\mu_2+\cdots+\mu_n=d$, when there are exactly $\mu_i$ 1's in the $i^{th}$ row, and so, the number of such matrices is $\frac{d!}{\mu_1!\mu_2!\cdots\mu_n!}=: {d\choose \mu}$. Therefore, we have:
\begin{equation} \label{eq:dmuv}
\chi(V_{n,d}) = \sum_{|\mu|=d} {d \choose \mu} v^{\mu},\end{equation}
By a similar argument, we can compute the character of the mirabolic tensor space $MV_{n,d}$ using the decorated matrix basis: a decorated matrix $m_{\bfr,\bfj}$ with $k$ decorations and weight $\mu$ can be chosen by deciding which $k$ columns and rows the decorated $1$'s will be in, and then choosing the other $d-k$ columns arbitrarily. Therefore,

\[\chi(MV_{n,d})=\sum_{|\mu|=d}\sum_{k=0}^{\min(n,d)}\sum_{\eps}{d-k \choose \mu-\eps}{d \choose k}  v^{\mu}.\]
Here, the last sum ranges over all $\eps\in\{0,1\}^n$ with exactly $k$ entries equal to 1, such that all the entries of $\mu-\eps$ are non-negative, and corresponds to the rows in which the decorations exist. The above sum can be rewritten as:
\[\chi(MV_{n,d})=\sum_{k=0}^{\min(n,d)}\sum_{|\mu|=d-k}\sum_{\eps}{d-k \choose \mu}{d \choose k}  v^{\mu+\eps},\]
where $\eps$ ranges over the entire set of sequences in $\{0,1\}^n$ with $k$ entries equal to $1$. Then, we can write:
\begin{align*}
\chi(MV_{n,d})&=\sum_{k=0}^{\min(n,d)}\sum_{|\mu|=d-k}\sum_{\eps}{d-k \choose \mu}{d \choose k}  v^{\mu+\eps}\\
&=\sum_{k=0}^{\min(n,d)}{d \choose k}\sum_{|\mu|=d-k}{d-k \choose \mu}  v^{\mu}\sum_{\eps}v^{\eps}\\
&=\sum_{k=0}^{\min(n,d)}{d \choose k}\chi(V_{n,d-k})\sum_{\eps}v^{\eps},
\end{align*}
where the last equality follows from Equation~\eqref{eq:dmuv}. Next, we compute the character of the right side of Equation~\eqref{eq:leftmod}.
\begin{align*}
\chi(\bigoplus_{\bs\lambda\in M\Lambda_{n,d}} L_{\bs\lambda}\otimes M^{\bs\lambda})&= \sum_{\bs\lambda\in M\Lambda_{n,d}} \chi(L_{\bs\lambda})\dim(M^{\bs\lambda})\\
&=\sum_{k=0}^{\min(n,d)}\sum_{\lambda\in\Lambda_{n,d-k}}\chi(L_{\lambda}\otimes W_k)\dim(M^{(\lambda,1^k)})\\
&=\sum_{k=0}^{\min(n,d)}\sum_{\lambda\in\Lambda_{n,d-k}}\chi(L_{\lambda})\chi(W_k)\dim(M^{\lambda}){d \choose k}\\
&=\sum_{k=0}^{\min(n,d)}{d \choose k}\sum_{\lambda\in\Lambda_{n,d-k}}\chi(L_{\lambda})f^{\lambda}\sum_{\eps}v^{\eps}.\\
\end{align*}
Here, we're using the fact that $\chi(W_k)=\sum_{\eps} v^{\eps}$ where $\eps$ ranges over the entire set of sequences in $\{0,1\}^n$ with $k$ entries equal to $1$. The equality of the two character expressions obtained above follows from Equation~\eqref{eq:chillam}.

In order to complete the proof of the proposition, we do a slightly refined version of the above computation in order to compute the character of both sides of Equation~\eqref{eq:leftmod} when restricting to the $\ell=1$ eigenspace. (The $\ell=1$ eigenspace isn't necessarily a $\uvsln$-representation, but it is preserved under the action of the $k_i$'s since they commute with $\ell$.) We get almost the same character expressions for both sides, with the only difference being that $\eps$ ranges over the set of sequences in $\{0,1\}^n$ with $k$ entries equal to $1$ and first entry necessarily being $0$. (This follows by looking at the basis of the $\ell=1$ eigenspace computed in Corollary~\ref{cor:lconse}.) This shows that the $(k_i,\ell=1)$-eigenspaces of both sides of Equation~\eqref{eq:leftmod} have the same dimension, thus completing the proof.
\end{proof}

\section{Proof of Theorem~\ref{thm:MSWD}} \label{sec:MSWDproof}

In section \ref{sec:leftmod}, we showed that the decomposition in~\eqref{eq:MSWD} is true as left $MU(n)$-modules. In particular, this implies that for a given bipartition $\bs\lambda$, the isotypic component of $L_{\bs\lambda}$ in $MV_{n,d}$ is non-zero if and only if $\bs\lambda\in M\Lambda_{n,d}$.

Hence, in order to prove Theorem~\ref{thm:MSWD}, we need to show that each $L_{\bs\lambda}$ indeed pairs with $M^{\bs\lambda}$ in $MV_{n,d}$. That is, as a right $MH_d$-module, the $L_{\bs\lambda}$-isotypic component of $MV_{n,d}$ exactly decomposes into copies of the module $M^{\bs\lambda}$. Note that we have already verified this for bipartitions $\bs\lambda=(\lambda,1^s)$ whenever $s=0$ in Proposition~\ref{prop:s=0}.
We will prove the complete correspondence using a reverse induction on $s$. The base case is proved in Section~\ref{sec:base} and the induction argument is in Section~\ref{sec:ind}.

\subsection{Simultaneous eigenspaces of the mirabolic Jucys-Murphy elements} \label{sec:zeroeigen}
In this section, we prove some results about the simultaneous zero eigenspaces of the mirabolic Jucys-Murphy elements, which will form the technical heart of the induction proof. Fix $\mathbf{r}=(r_1,r_2,\dots,r_d)\in S_{row}.$ Let $\Sigma_{\mathbf{r}}\subseteq J_{\bfr}$ denote the set of \emph{good decorations}, which is defined as
\[\Sigma_{\mathbf{r}}:=\{\bfj\in J_{\bfr}~|~a\in\bfj\implies r_a<r_k,~\forall 1\leq k\leq a-1\}. \]
In particular, if $\mathbf{j}\in\Sigma_{\bfr}$ is a good decoration, then $\{1\}\cup \mathbf{j}\in \Sigma_{\bfr}$. 

Recall that for any permutation $\sigma\in \mathcal{S}_n$, we define $\len(\sigma)$ to be the length of $\sigma$, which is equal to be the number of inversions in $\sigma$.

\begin{lemma} \label{lem:simzero}
Suppose $n\leq d$. Fix $\mathbf{r}=(n, n-1, \dots, 2, 1, r_{n+1},r_{n+2},\cdots,r_d)$, where $1\leq r_i\leq n$. Then, the vector:
\[\sum_{\sigma\in \mathcal{S}_n} \sum_{\mathbf{j}\in \Sigma_{\sigma(\mathbf{r})}} (-v^2)^{-\len(\sigma)}(1-v^2)^{n-|\mathbf{j}|}m_{\sigma(\mathbf{r}),\mathbf{j}},\]
 lies in the simultaneous zero eigenspace of $\LL_1, \LL_2, \dots, \LL_n$. Furthermore, the idempotent $\ell\in MU(n)$ acts on it by zero.
(The action of the symmetric group $\mathcal{S}_n$ on $\bfr$ comes from the embedding $\mathcal{S}_n\times\{1\}\hra \mathcal S_n\times \mathcal S_{d-n}\hra \mathcal S_d$, i.e. it permutes the first $n$ terms.)
\end{lemma}

\begin{remark}
The vector defined in the above lemma was first constructed by applying $\mathbf{e}_d$ (from Remark~\ref{rem:TL}) on the element $m_{\mathbf{r},\emptyset}$ of $MV_{d,d}$ for $d=2$.
\end{remark}

\begin{proof}
The fact that $\ell$ acts by zero is an immediate consequence of Corollary~\ref{cor:lconse}, since if $j=\sigma(n)$ is the minimal index in $\sigma(\bfr)$ with $r_j=1$, then both $\bfj,\bfj\setminus\{j\}\in\Sigma_{\sigma(\bfr)}$, and 
\begin{align*}&\ell\cdot\left((1-v^2)^{n-|\mathbf{j}|}m_{\sigma(\mathbf{r}),\mathbf{j}}+(1-v^2)^{n-|\mathbf{j}\setminus\{j\}|}m_{\sigma(\mathbf{r}),\mathbf{j}\setminus\{j\}}\right) \\
&~=(1-v^2)^{n-|\mathbf{j}|}(v^2-1)\ell\cdot  m_{\sigma(\mathbf{r}),\mathbf{j}\sm\{j\}}+(1-v^2)^{n+1-|\mathbf{j}|}\ell \cdot m_{\sigma(\mathbf{r}),\mathbf{j}\sm\{j\}}=0.\end{align*}
To show that the action of the $\LL_i$'s is zero, it suffices to show that $e$ acts by $0$ and $\tau_j$ acts by $-1$ for all $1\leq j\leq n-1$. To prove the former, we note that, since $\bfj\cup\{1\}\in\Sigma_{\sigma(r)}$ for all $\bfj\in\Sigma_{\sigma(r)}$, we have for all $\sigma\in\mathcal{S}_n$:
\begin{align*}
&\sum_{\mathbf{j}\in \Sigma_{\sigma(\mathbf{r})}} (-v^2)^{-\len(\sigma)}(1-v^2)^{n-|\mathbf{j}|}m_{\sigma(\mathbf{r}),\mathbf{j}}\\
=& \sum_{\substack{\mathbf{j}\in \Sigma_{\sigma(\mathbf{r})},\\ 1\in\mathbf{j}}} (-v^2)^{-\len(\sigma)}(1-v^2)^{n-|\mathbf{j}|}m_{\sigma(\mathbf{r}),\mathbf{j}}+ \sum_{\substack{\mathbf{j}\in \Sigma_{\sigma(\mathbf{r})},\\ 1\not\in\mathbf{j}}} (-v^2)^{-\len(\sigma)}(1-v^2)^{n-|\mathbf{j}|}m_{\sigma(\mathbf{r}),\mathbf{j}}\\
=& \sum_{\substack{\mathbf{j}\in \Sigma_{\sigma(\mathbf{r})},\\ 1\in\mathbf{j}}} (-v^2)^{-\len(\sigma)}(1-v^2)^{n-|\mathbf{j}|}\big(m_{\sigma(\mathbf{r}),\mathbf{j}}-(v^2-1)m_{\sigma(\mathbf{r}),\mathbf{j}\sm\{1\}}\big).
\end{align*}
It is clear that $e$ acts by zero on each summand of the above summation by Corollary~\ref{cor:eeigen}, showing that $e$ acts by zero on the whole sum.

Next, we consider the action of $\tau_j$. Fix $1\leq b<a\leq n$ and let $\mathcal S_{a,b}$ be the set of permutations $\sigma$ such that $\sigma(j)=a$ and $\sigma(j+1)=b$. Also, let $s_j$ be the transposition $(j,j+1)$. Then $\mathcal S_{a,b}\cdot s_j$ is the set of those $\sigma$ such that $\sigma(j)=b$ and $\sigma(j+1)=a$. (Notice that for each $\sigma\in\mathcal S_{a,b}$, we have $\len(\sigma s_j)=\len(\sigma)-1$. For each pair $(a,b)$ as above and $\sigma\in\mathcal S_{a,b}$, it suffices to show that $\tau_j$ acts by $-1$ on the sum:
\begin{equation} \label{eq:sumsigma}
\sum_{\mathbf{j}\in \Sigma_{\sigma(\mathbf{r})}} (-v^2)^{-\len(\sigma)}(1-v^2)^{n-|\mathbf{j}|}m_{\sigma(\mathbf{r}),\mathbf{j}}+\sum_{\mathbf{j}\in \Sigma_{\sigma\circ s_j(\mathbf{r})}} (-v^2)^{-\len(\sigma s_j)}(1-v^2)^{n-|\mathbf{j}|}m_{\sigma\circ s_j(\mathbf{r}),\mathbf{j}}.
\end{equation}
We will show this via an explicit computation using the formulae for the $\tau_j$-action from Section~\ref{sec:action-formulae}.

Let $\widetilde{\mathbf{j}}$ be a good decoration of $m_{\sigma(\mathbf{r})}$ such that $j,j+1\not\in\widetilde{\mathbf{j}}$. We consider those summands of the sum~\eqref{eq:sumsigma} such that $\bfj\setminus \widetilde{\bfj}\subset\{j,j+1\}$. (That is, we consider those decorations $\bfj$ such that $\bfj$ and $\widetilde\bfj$ coincide outside the $j^{th}$ and $(j+1)^{th}$ columns.)

For the rest of the argument, we will be viewing $m_{\mathbf{\sigma(r)},\mathbf{j}}$ as a decorated $n\times d$ matrix. Furthermore, we will be focusing on the $2\times 2$ submatrix consisting of entries from the $j^{th}$ and $(j+1)^{th}$ columns and the $a^{th}$ and $b^{th}$ rows. So, when we write the matrix $\begin{pmatrix}0 & 1 \\ 1 & 0\end{pmatrix}$ we are representing a matrix of the form \begin{align*}&\left( \begin{array}{cccc}  & \vdots & \vdots &  \\ \hdots & 0 & 1 & \hdots \\  & \vdots & \vdots & \\ \hdots & 1 & 0 & \hdots \\ & \vdots & \vdots & \end{array}\right)\begin{array}{c} ~ \\ \scriptstyle{b} \\ \phantom{\vdots} \\ \scriptstyle{a} \\ ~ \end{array}\\
&\begin{array}{cccccc} & \phantom{\hdots} & \scriptstyle{j} & \scriptstyle{j+1} & & \end{array}\end{align*}
Note that the right action of an element $\tau_j$ only affects the $2\times 2$ submatrix, so these are sufficient for the computations that follow.

Let $\sigma(\mathbf{r})=(r_1,r_2,\dots,r_d)$. We have three cases:
\begin{enumerate} [label=\alph*)]
\item $r_k<b<a$ for some $1\leq k\leq j-1$: In this case, for any $\bfj\in\Sigma_{\sigma(\bfr)}$ or $\bfj\in\Sigma_{\sigma\circ  s_j(\bfr)}$, we have $j,j+1\not\in\bfj$. Hence, the summands of the sum~\eqref{eq:sumsigma} such that $\bfj\setminus \widetilde{\bfj}\subset\{j,j+1\}$ are:
\[(-v^2)^{-\len(\sigma)}(1-v^2)^{n-|\widetilde{\mathbf{j}}|}\Bigg(\begin{pmatrix}0&1\\1&0\end{pmatrix}-v^2\begin{pmatrix}1&0\\0&1\end{pmatrix}\Bigg).\]
Then,
\begin{align*}
\Bigg(\begin{pmatrix}0&1\\1&0\end{pmatrix}-v^2\begin{pmatrix}1&0\\0&1\end{pmatrix}\Bigg)\cdot \tau_j&=\begin{pmatrix}0&1\\1&0\end{pmatrix}\cdot \tau_j-v^2\begin{pmatrix}1&0\\0&1\end{pmatrix}\cdot \tau_j\\
&=v^2\begin{pmatrix}1&0\\0&1\end{pmatrix}+(v^2-1)\begin{pmatrix}0&1\\1&0\end{pmatrix} -v^2\begin{pmatrix}0&1\\1&0\end{pmatrix}\\
&=v^2\begin{pmatrix}1&0\\0&1\end{pmatrix}-\begin{pmatrix}0&1\\1&0\end{pmatrix},
\end{align*}
showing that $\tau_j$ acts by $-1$.
\item $a>r_k$ for some $1\leq k\leq j-1$ but $b<r_k$ for all $1\leq k\leq j-1$: In this case, for any $\bfj\in\Sigma_{\sigma(\bfr)}$ or $\bfj\in\Sigma_{\sigma \circ s_j(\bfr)}$, we have $j\in\bfj$ (respectively $j+1\in\bfj$) only if $r_j=b$ (respectively $r_{j+1}=b$). So, the summands of the sum~\eqref{eq:sumsigma} such that $\bfj\setminus \widetilde{\bfj}\subset\{j,j+1\}$ are:
\[(-v^2)^{-\len(\sigma)}(1-v^2)^{n-|\widetilde{\mathbf{j}}|}\Bigg(\begin{pmatrix}0&1\\1&0\end{pmatrix}-(v^2-1)^{-1}\begin{pmatrix}0&\circled1\\1&0\end{pmatrix}-v^2\begin{pmatrix}1&0\\0&1\end{pmatrix}+v^2(v^2-1)^{-1}\begin{pmatrix}\circled1&0\\0&1\end{pmatrix}\Bigg).\]
Then,
\begin{align*}
&\Bigg((v^2-1)\begin{pmatrix}0&1\\1&0\end{pmatrix}- \begin{pmatrix}0&\circled1\\1&0\end{pmatrix}-v^2(v^2-1)\begin{pmatrix}1&0\\0&1\end{pmatrix}+v^2\begin{pmatrix}\circled1&0\\0&1\end{pmatrix}\Bigg)\cdot \tau_j\\
&=(v^2-1)\begin{pmatrix}0&1\\1&0\end{pmatrix}\cdot \tau_j-\begin{pmatrix}0&\circled1\\1&0\end{pmatrix}\cdot \tau_j-v^2(v^2-1)\begin{pmatrix}1&0\\0&1\end{pmatrix}\cdot \tau_j+v^2\begin{pmatrix}\circled1&0\\0&1\end{pmatrix}\cdot \tau_j\\
&=v^2(v^2-1)\begin{pmatrix}1&0\\0&1\end{pmatrix}+(v^2-1)^2\begin{pmatrix}0&1\\1&0\end{pmatrix} - v^2\begin{pmatrix}\circled1&0\\0&1\end{pmatrix}-(v^2-1)\begin{pmatrix}0&\circled1\\1&0\end{pmatrix}\\
&-v^2(v^2-1)\begin{pmatrix}0&1\\1&0\end{pmatrix}+v^2\begin{pmatrix}0&\circled1\\1&0\end{pmatrix}\\
&=v^2(v^2-1)\begin{pmatrix}1&0\\0&1\end{pmatrix}-(v^2-1)\begin{pmatrix}0&1\\1&0\end{pmatrix}- v^2\begin{pmatrix}\circled1&0\\0&1\end{pmatrix}+\begin{pmatrix}0&\circled1\\1&0\end{pmatrix},
\end{align*}
showing that $\tau_j$ acts by $-1$.
\item $b<a<r_k$ for all $1\leq k\leq j-1$: In this case, we get a sum over all valid good decorations of the $j^{th}$ and $(j+1)^{th}$ columns. So, the summands of the sum~\eqref{eq:sumsigma} such that $\bfj\setminus \widetilde{\bfj}\subset\{j,j+1\}$ are:
\begin{align*}
(-v^2)^{-\len(\sigma)}(1-v^2)^{n-|\widetilde{\mathbf{j}}|}\Bigg(&\begin{pmatrix}0&1\\1&0\end{pmatrix}-(v^2-1)^{-1}\begin{pmatrix}0&\circled1\\1&0\end{pmatrix}-(v^2-1)^{-1}\begin{pmatrix}0&1\\\circled1&0\end{pmatrix}\\
&+(v^2-1)^{-2}\begin{pmatrix}0&\circled1\\\circled1&0\end{pmatrix}-v^2\begin{pmatrix}1&0\\0&1\end{pmatrix}+v^2(v^2-1)^{-1}\begin{pmatrix}\circled1&0\\0&1\end{pmatrix}\Bigg).
\end{align*}
Then,
\begin{align*}
&\Bigg((v^2-1)^2\begin{pmatrix}0&1\\1&0\end{pmatrix}-(v^2-1)\begin{pmatrix}0&\circled1\\1&0\end{pmatrix}-(v^2-1)\begin{pmatrix}0&1\\\circled1&0\end{pmatrix}\\
&+\begin{pmatrix}0&\circled1\\\circled1&0\end{pmatrix}-v^2(v^2-1)^2\begin{pmatrix}1&0\\0&1\end{pmatrix}+v^2(v^2-1)\begin{pmatrix}\circled1&0\\0&1\end{pmatrix}\Bigg)\cdot \tau_j\\
&=(v^2-1)^2\begin{pmatrix}0&1\\1&0\end{pmatrix}\cdot \tau_j-(v^2-1)\begin{pmatrix}0&\circled1\\1&0\end{pmatrix}\cdot \tau_j-(v^2-1)\begin{pmatrix}0&1\\\circled1&0\end{pmatrix}\cdot \tau_j\\
&+\begin{pmatrix}0&\circled1\\\circled1&0\end{pmatrix}\cdot \tau_j-v^2(v^2-1)^2\begin{pmatrix}1&0\\0&1\end{pmatrix}\cdot \tau_j+v^2(v^2-1)\begin{pmatrix}\circled1&0\\0&1\end{pmatrix}\cdot \tau_j\\
&=v^2(v^2-1)^2\begin{pmatrix}1&0\\0&1\end{pmatrix}+(v^2-1)^3\begin{pmatrix}0&1\\1&0\end{pmatrix}- v^2(v^2-1)\begin{pmatrix}\circled1&0\\0&1\end{pmatrix}-(v^2-1)^2\begin{pmatrix}0&\circled1\\1&0\end{pmatrix}\\
&-(v^2-1)\begin{pmatrix}1&0\\0&\circled1\end{pmatrix}-(v^2-1)\begin{pmatrix}0&\circled1\\\circled1&0\end{pmatrix}+(v^2-1)\begin{pmatrix}1&0\\0&\circled1\end{pmatrix}+(v^2-1)\begin{pmatrix}0&1\\\circled1&0\end{pmatrix}\\
&+(v^2-2)\begin{pmatrix}0&\circled1\\\circled1&0\end{pmatrix}-v^2(v^2-1)^2\begin{pmatrix}0&1\\1&0\end{pmatrix}+v^2(v^2-1)\begin{pmatrix}0&\circled1\\1&0\end{pmatrix}\\
&=v^2(v^2-1)^2\begin{pmatrix}1&0\\0&1\end{pmatrix}-(v^2-1)^2\begin{pmatrix}0&1\\1&0\end{pmatrix}- v^2(v^2-1)\begin{pmatrix}\circled1&0\\0&1\end{pmatrix}+(v^2-1)\begin{pmatrix}0&\circled1\\1&0\end{pmatrix}\\
&-\begin{pmatrix}0&\circled1\\\circled1&0\end{pmatrix}+(v^2-1)\begin{pmatrix}0&1\\\circled1&0\end{pmatrix},
\end{align*}
showing that $\tau_j$ acts by $-1$.
\end{enumerate}
\end{proof}



We now find bases for simultaneous zero eigenspaces of the Jucys-Murphy elements by slightly generalizing the above lemma. Fix $k\leq \min(n,d)$ and let $S\in [n]_k$ be a $k$-element subset of $\{1,2,\dots, n\}$. By abuse of notation, we will also view $S$ as a sequence, by arranging the elements of $S$ in decreasing order. If $S=(r_1,\ldots,r_k)$ (with $r_1>\cdots>r_k$), and $\sigma\in \mathcal{S}_k$, we can consider the sequence $\sigma(S)=(r_{\sigma^{-1}(1)},\ldots,r_{\sigma^{-1}(k)})$. We define: \[\Gamma_S:=\{(\sigma, \bfj): \sigma\in \mathcal S_k, \bfj\in \Sigma_{\sigma(S)}\}.\]
We also recall from~\eqref{eq:jrs-def} the set $J_{\bfr}(S)=\{\bfj\in J_{\bfr}~|~\forall j\in\bfj~\text{and}~s\in S,~r_j< s\}$. For any such set $X\sub \ZZ$ and $k\in \ZZ$, we define:
\[X+k:=\{x+k:x\in X\}.\]

\begin{definition}
Given $S\in [n]_k$, an element ${\mathbf{r}}\in\{1,2,\dots,n\}^{d-k}$ and a decoration ${\mathbf{j}}\in J_{\bfr}(S)$, we define the element:
\[\mathbf{s}_{S,{\mathbf{r}},{\mathbf{j}}}:=\sum_{(\sigma,\widetilde{\mathbf{j}})\in \Gamma_S}  (-v^2)^{-\len(\sigma)}(1-v^2)^{k-|\widetilde{\mathbf{j}}|}m_{\sigma(S)\cup {\mathbf{r}},\widetilde{\mathbf{j}}\cup({\mathbf{j}}+k)}.\]
(In the above formula, `$\sigma(S)\cup{\bfr}$'  refers to the concatenation of sequences, whereas `$\widetilde{\mathbf{j}}\cup({\mathbf{j}}+k)$' refers to a union of sets.)
\end{definition}

\begin{remark} \label{rem:s-lem41}
Note that when $k=n\leq d$, we can take $S=\{1,2,\dots,n\}$, arbitrary $\bfr\in\{1,2,\dots,n\}^{d-n}$ and $\bfj=\emptyset$. In that case, the element $\mathbf{s}_{S,{\mathbf{r}},{\mathbf{j}}}$ is exactly the vector defined in Lemma~\ref{lem:simzero}.
\end{remark}


\begin{definition}
Let $\mathbf{r}\in \{1,2,\dots, n\}^d$ and fix $k$ such that $1\leq k\leq \min(n,d)$. We define the $k$-support of $m_{\mathbf{r}}$ to be the set $Supp_k(m_{\mathbf{r}}):=\{r_i:1\leq i\leq k\}$. For any $v=\sum_{\bfr} c_{\bfr}m_{\bfr}$ for some scalars $c_{\bfr}$, we define:
\[Supp_k(v):=\bigcup_{c_{\bfr}\neq 0}Supp_k(m_{\bfr}).\]
\end{definition}

It follows that $|Supp_k(m_{\mathbf{r}})|\leq k$ for all $\mathbf{r}$. Note that for all the summands in the definition of $\mathbf{s}_{S,{\mathbf{r}},{\mathbf{j}}}$, the $k$-support is equal to $S$. (In particular, the support set has size $k$.) 

\begin{remark} \label{rem:chopoff}
\begin{enumerate}
\item The definition of support is motivated by the following fact which follows from the formulae for the $MH_d$-action: Given $1\leq k\leq \min(n,d)$ and $S\in [n]_k$, let $V_S$ denote the subspace of $MV_{n,d}$ spanned by all $m_{\bfr,\bfj}$'s whose $k$-support lies in $S$ and and any decorations in the last $d-k$ columns lie above all the rows in $S$. Then, the linear map $V_S\to MV_{k,k}$ obtained by deleting the last $d-k$ columns and all the rows outside $S$ is equivariant with respect to the $MH_k$-action. Here, we view the algebra $MH_k$ as the subalgebra of $MH_d$ generated by $e$ and $\tau_i$ for $1\leq i\leq k-1$.

\item Similarly, given such $k$ and $S$, consider a sequence $\widetilde \bfr=(r_1,r_2,\dots, r_i)$ such that $\min(S)\leq r_1<r_2<\dots <r_i\leq \max(S)$. Let $V_{S,\widetilde \bfr}$ denote the subspace spanned by all $m_{\bfr,\bfj}$'s whose $k$-support lies in $S$ and having a decoration somewhere in the last $d-k$ columns in the $r_j^{th}$ row for each $j$. (Note that if $i>d-k$, this subspace is zero.) Then, we obtain an $MH_k$-equivariant map $V_{S,\widetilde \bfr}\to MV_{k,k+1}$ by deleting the last $d-k-1$ columns and all the rows outside $S$, and replacing the $(k+1)^{th}$ column with one having a decorated $1$ in the row indexed by $\max([1,r_i]\cap S)$.
\end{enumerate}
\end{remark}

\begin{lemma} \label{lem:zerobasis}
Let $1\leq k\leq\min(n,d)$ and let $Z_k$ denote the simultaneous zero eigenspace for $\LL_1, \LL_2,\dots, \LL_k$ in $MV_{n,d}$. (That is, $Z_k=\bigcap_{i=1}^k\ker(\LL_i)$.) Then, a basis for $Z_k$ is given by the elements $\mathbf{s}_{S,{\mathbf{r}},{\mathbf{j}}}$, for $S\in[n]_k$, ${\bfr}\in\{1,\ldots,n\}^{d-k}$, ${\bfj}\in J_{{\bfr}}(S)$. Furthermore, if $n<k'\leq d$, we have $\bigcap_{i=1}^{k'}\ker(\LL_i)=0$.
\end{lemma}


\begin{remark}
By Theorems~\ref{thm:MHd-reps} and \ref{thm:JM_action}, any $MH_d$-representation (in particular, $MV_{n,d}$) has a basis with respect to which the Jucys-Murphy elements are simultaneously diagonalizable. Then, as a consequence of the above theorem and the description of eigenvalues in Theorem~\ref{thm:JM_action}, we have that that $\LL_{n+1}$ acts by $1$ on the space $Z_n\sub MV_{n,d}$.
\end{remark}

The action of the Jucys-Murphy elements being zero on $\mathbf{s}_{S,{\mathbf{r}},{\mathbf{j}}}$ follows from Lemma~\ref{lem:simzero} and Remark~\ref{rem:chopoff}(1). Linear independence is clear since the non-zero summands in each $\mathbf{s}_{S,{\mathbf{r}},{\mathbf{j}}}$ are distinct. The crucial step is to show that they indeed span $Z_k.$ The proof will be using induction on $k$. The base case of the induction is as follows. When $k=1$, the zero eigenspace of $\LL_1=e$ is spanned by:
\[\{m_{\mathbf{r},\mathbf{j}}-(v^2-1)m_{\mathbf{r},\mathbf{j}\setminus\{1\}}~|~1\in\mathbf{j}\}\]
by Corollary~\ref{cor:eeigen}. Then, the claim follows from the fact that:
\[\mathbf{s}_{S,\widetilde{\bfr},\widetilde{\bfj}}=m_{\mathbf{r},\mathbf{j}}-(v^2-1)m_{\mathbf{r},\mathbf{j}\setminus\{1\}},\]
where $S=\{r_1\}, \widetilde{\bfr}=(r_2,\ldots,r_d)$ and $\widetilde{\bfj}=(\bfj\setminus\{1\})-1$. To complete the proof of the base case, we show that when $n=1$, then $Z_1\cap \ker(\LL_2)=0$ (i.e., $\LL_1$ and $\LL_2$ do not have a simultaneous zero eigenvector). To this end, we note that when $n=1$, the space $Z_1$ is $1$-dimensional, and is spanned by 
\[\begin{pmatrix} \circled1 & 1 & \cdots & 1 \end{pmatrix}-(v^2-1)\begin{pmatrix}1 & 1 & \cdots & 1 \end{pmatrix}.\]
Then, we have  \begin{align*}
&\Big(\begin{pmatrix} \circled1 & 1 & \cdots & 1 \end{pmatrix}-(v^2-1)\begin{pmatrix}1 & 1 & \cdots & 1 \end{pmatrix}\Big)\cdot \LL_2\\
=&\Big(\begin{pmatrix} \circled1 & 1 & \cdots & 1 \end{pmatrix}-(v^2-1)\begin{pmatrix}1 & 1 & \cdots & 1 \end{pmatrix}\Big)\cdot v^{-2}\tau_1 e\tau_1 \\
=&\Big(\begin{pmatrix} 1& \circled1 & \cdots & 1 \end{pmatrix}-(v^2-1)v^2\begin{pmatrix}1 & 1 & \cdots & 1 \end{pmatrix}\Big)\cdot v^{-2}e\tau_1 \\
=&\Big(\begin{pmatrix} 1& \circled1 & \cdots & 1 \end{pmatrix}-(v^2-1)\big(\begin{pmatrix}1 & 1 & \cdots & 1 \end{pmatrix}+\begin{pmatrix}\circled1 & 1 & \cdots & 1 \end{pmatrix}\big)\Big)\cdot v^{-2}\tau_1 \\
=&\Big(v^2\begin{pmatrix} \circled1 & 1&  \cdots & 1 \end{pmatrix}+(v^2-1)\begin{pmatrix} 1& \circled1 & \cdots & 1 \end{pmatrix}+ \\
& -(v^2-1)v^2\begin{pmatrix}1 & 1 & \cdots & 1 \end{pmatrix}-(v^2-1)\begin{pmatrix} 1 & \circled1 & \cdots & 1 \end{pmatrix}\Big)v^{-2}\\
=&\begin{pmatrix} \circled1 & 1 & \cdots & 1 \end{pmatrix}-(v^2-1)\begin{pmatrix}1 & 1 & \cdots & 1 \end{pmatrix},
\end{align*}
showing that $\LL_2$ acts by $1$ on $Z_1$.

Before we work with general $k$, we prove two special cases (Lemmas~\ref{lem:knd} and~\ref{lem:knd-1}) of the above lemma, and then show that the general case can be reduced to these two cases.


\begin{lemma} \label{lem:knd}
Let $Z_k=\bigcap_{i=1}^k\ker(\LL_i)$ in $MV_{k,k}$. Then, the space $Z_k$ is one-dimensional and is spanned by the vector $\mathbf{s}_{S,\emptyset,\emptyset}$, where $S=\{1,2,\dots, k\}$.
\end{lemma}

\begin{proof}
In the space $MV_{k,k}$, the space $Z_k$ is exactly the isotypic component of the one-dimensional $MH_k$-representation corresponding to the bipartition $(\emptyset,1^k)$. By Proposition~\ref{prop:leftmod}, the only $L_{\bs\lambda}$'s that pair with a one-dimensional representation of $MH_k$ correspond to $\bs\lambda=(k,\emptyset)$, $\bs\lambda=(1^k,\emptyset)$ or $\bs\lambda=(\emptyset,1^k)$. Each of the $\LL_i$'s act with non-zero eigenvalues in the first two cases by Proposition~\ref{prop:s=0}. Thus, the representation $M^{(\emptyset,1^k)}$ of $MH_k$ must pair with the $L_{\bs\lambda}$ when $\bs\lambda=(\emptyset,1^k)$, which is itself a one-dimensional representation. Hence, the space $Z_k$ is one dimensional and is spanned by the vector $\mathbf{s}_{S,\emptyset,\emptyset}$, where $S=\{1,2,\dots, k\}$. 
\end{proof}

\begin{lemma} \label{lem:knd-1}
Let $Z_k=\bigcap_{i=1}^k\ker(\LL_i)$ in $MV_{k,k+1}$. Then, the space $Z_k$ is $k$-dimensional and is spanned by the vectors $\mathbf{s}_{S,{\mathbf{r}},{\mathbf{j}}}$ where $S=\{1,2,\dots,k\}$, ${\mathbf{r}}\in\{1,2,\dots, k\}$ and ${\mathbf{j}}=\emptyset$. Furthermore, $\bigcap_{i=1}^{k+1}\ker(\LL_i)=0$.

In particular, each of these basis vectors has $k$-support of size $k$ and no summands with decorations in the last column.
\end{lemma}

\begin{proof}
We first show that $\bigcap_{i=1}^{k+1}\ker(\LL_i)=0$. Any non-zero vector in this intersection would span the one-dimensional representation $M^{(\emptyset,1^{k+1})}$ of $MH_{k+1}$. By Proposition~\ref{prop:leftmod}, the only $L_{\bs\lambda}$'s in $MV_{k,k+1}$ that pair with one-dimensional $MH_{k+1}$-representations correspond to $\bs\lambda=(k+1,\emptyset)$ (since $\bfla=(1^{k+1},\emptyset)$ and $\bfla=(\emptyset,1^{k+1})$ both contain a partition with more than $k$ parts).
However, by Proposition~\ref{prop:s=0}, none of the Jucys-Murphy elements act by zero on the isotypic component of this $\bs\lambda$. This proves the required claim.

For the rest of the lemma, we proceed by induction on $k$. For the base case when $k=1$, we need to show that the space $Z_1\sub MV_{1,2}$ is spanned by the element $\mathbf{s}_{S,{\mathbf{r}},{\mathbf{j}}}$ where $S=\{1\}$, ${\mathbf{r}}=\{1\}$ and ${\mathbf{j}}=\emptyset$ and that $\LL_2$ does not kill this vector. The base case of Lemma~\ref{lem:zerobasis} already established both of these facts.

The action of the first $k$ Jucys-Murphy elements on a decorated matrix leaves the last ($=(k+1)^{th}$) column invariant. Hence, we can write a decomposition of the space $Z_k$ depending on where the $1$ exists in the last column and whether it is decorated or not. More concretely, we can write:
\[Z_k=\bigoplus_{\theta}{Z_{\theta}},\]
where $\theta$ ranges over pairs $(\widetilde\bfr,\widetilde\bfj)$ and $Z_{\theta}$ is the subspace of $Z_k$ spanned by $m_{\bfr,\bfj}$'s where $\widetilde\bfr=\{i\}$ (with $1\leq i\leq k$) denotes the position of the 1 in the last column and $\widetilde\bfj$ is equal to $\{1\}$ or $\emptyset$ depending on whether the last column is decorated or not, respectively.

For pairs of the form $\theta=(\{i\},\emptyset)$, the space $Z_{\theta}$ can be embedded into the tensor space $MV_{k,k}$ by simply deleting the last column. This embedding respects the $MH_k$-action. In particular, the space $Z_{\theta}$ lies in the simultaneous zero eigenspace of the first $k$ Jucys-Murphy elements in $MV_{k,k}$. By the previous lemma, this space is one-dimensional showing that $\dim(Z_{\theta})=1$. Thus, we have:
\[\sum_{\theta=(\{i\},\emptyset)}\dim(Z_\theta) = k,\]
and so, to prove the lemma, it suffices to show that $Z_{\theta}=0$ whenever $\theta=(\{i\},\{1\})$ for any $i$.

So, suppose $\theta=(\{i\},\{1\})$. On the space $Z_{\theta}$, the first $k$, and more specifically the first $k-1$ Jucys-Murphy elements act by zero. By induction and Remark~\ref{rem:chopoff}(2), any decorations in the last two columns must lie in a row above the $(k-1)$-support of a simultaneous zero eigenvector for the first $k-1$ Jucys-Murphy elements and the $(k-1)$-support has size $k-1$. 

Finally, we suppose that $\theta=(\{1\},\{1\})$. 
By the discussion in the previous paragraph, all non-zero elements of $Z_{\theta}$ must have $(k-1)$-support equal to $\{k,k-1,\dots, 2\}$. By induction and Remark~\ref{rem:chopoff}(1), since $\LL_1, \LL_2, \dots, \LL_k$ do not have a simultaneous zero eigenvector in $MV_{k-1,k}$, the $1$ in the $k^{th}$ column of (any summand of) a non-zero element of $Z_{\theta}$ must lie in the first row. Thus, the space $Z_{\theta}$ is at most $1$-dimensional and (if it is non-zero) is spanned by the vector $\mathbf{s}_{S,{\bfr},{\bfj}}$ with $S=(k,k-1,\dots, 2)$, ${\bfr}=(1,1)$ and ${\bfj}=\{2\}$. We claim that $\mathbf{s}_{S,{\bfr},{\bfj}}\cdot \LL_{k}\neq 0$, which will show that $Z_\theta=0$.

To this end, we recall that $\LL_k=v^{-2}\tau_{k-1}\LL_{k-1}\tau_{k-1}$ and suppose, for the sake of contradiction, that $\mathbf{s}_{S,{\bfr},{\bfj}}\cdot \LL_{k}= 0$. Then, we must have $\mathbf{s}_{S,{\bfr},{\bfj}}\cdot \tau_{k-1}\LL_{k-1}= 0$ since $\tau_{k-1}$ is invertible. Furthermore, since $\LL_j$ commutes with $\tau_{k-1}$ for all $j<k-1$, we have that $\mathbf{s}_{S,{\bfr},{\bfj}}\cdot \tau_{k-1}\LL_{j}=\mathbf{s}_{S,{\bfr},{\bfj}}\cdot\LL_j\tau_{k-1}= 0$ since we already know that $\mathbf{s}_{S,{\bfr},{\bfj}}\cdot\LL_j=0$. Finally, using the fact that $\tau_{k-1}^2=(v^2-1)\tau_{k-1}+v^2$, we see that
\begin{align*}
\mathbf{s}_{S,{\bfr},{\bfj}}\cdot \tau_{k-1}\LL_{k}&=\mathbf{s}_{S,{\bfr},{\bfj}}\cdot v^{-2} \tau_{k-1}^2\LL_{k-1}\tau_{k-1} \\
&= \mathbf{s}_{S,{\bfr},{\bfj}}\cdot v^{-2} (v^2-1)\tau_{k-1}\LL_{k-1}\tau_{k-1}+\mathbf{s}_{S,{\bfr},{\bfj}}\cdot v^{-2} v^2 \LL_{k-1}\tau_{k-1} \\
&=(v^2-1) \mathbf{s}_{S,{\bfr},{\bfj}}\cdot \LL_k+\mathbf{s}_{S,{\bfr},{\bfj}}\cdot v^{-2} v^2 \LL_{k-1}\tau_{k-1}=0+0=0.
\end{align*}
Therefore, $\mathbf{s}_{S,{\bfr},{\bfj}}\cdot \tau_{k-1}$ is  a simultaneous zero eigenvector of $\LL_1, \LL_2, \dots, \LL_k$, and thus, $\mathbf{s}_{S,{\bfr},{\bfj}}\cdot \tau_{k-1}\in Z_\theta$ for $\theta=(\{1\},\{1\})$. However, since $Z_\theta$ is at most one-dimensional, this implies that $\mathbf{s}_{S,{\bfr},{\bfj}}$ must be an eigenvector for $\tau_{k-1}$. That this is not true follows immediately from the explicit formula for $\mathbf{s}_{S,{\bfr},{\bfj}}$. (For example, the action of $\tau_{k-1}$ swaps the $(k-1)^{th}$ and $k^{th}$ columns for some of the summands of  $\mathbf{s}_{S,{\bfr},{\bfj}}$, resulting in a $1$ in the first row of the $(k-1)^{th}$ column.)
\end{proof}

\begin{proof} [Proof of Lemma~\ref{lem:zerobasis}]
We use induction on $k$. The case $k=1$ has already been dealt with, and we now suppose that $k>1$ is arbitrary. By induction, the first $k$ Jucys-Murphy elements can not simultaneously act by zero on any non-zero vector of $MV_{k-1,d}$. Hence, by Remark~\ref{rem:chopoff}, any non-zero vector in $Z_k$ must have $k$-support of size at least $k$. The action of the first $k$ Jucys-Murphy elements on a decorated matrix leaves invariant the $(k+1)^{th}$ column and beyond. Hence, we can decompose the space $Z_k$ on the basis of $k$-support. That is, we can write:
\[Z_k=\bigoplus_{\theta} Z_{\theta},\]
where $\theta$ ranges over triples $(S,\widetilde{\bfr},\widetilde{\bfj})$, and $Z_{S,\widetilde{\bfr},\widetilde{\bfj}}$ is the subspace of $Z_k$ spanned by $m_{\mathbf{r},\bfj}$'s whose $k$-support is $S$, $\widetilde{\bfr}$ dictates the positions of $1$'s in the last $d-k$ columns and $\widetilde{\bfj}$ describes the set of decorations in the last $d-k$ columns. We fix one such $\theta$. As argued above, we can suppose that $|S|=k$ if $Z_{\theta}\neq 0$. Also, by Lemma~\ref{lem:knd-1} and Remark~\ref{rem:chopoff}(2), we can suppose that for any non-zero summand $m_{\bfr,\bfj}$ of any element in $Z_k$, all the decorations in the last $d-k$ columns lie above all the rows of the $(k-1)$-support.

If $\widetilde{\bfj}\in J_{\widetilde\bfr}(S)$, by Remark~\ref{rem:chopoff}(1), we can delete the last $d-k$ columns as well as any rows other than the ones indexed by $S$ and reduce to the case when $k=n=d$. Then, by Lemma~\ref{lem:knd}, we know that the space $Z_{(S,\widetilde\bfr,\widetilde\bfj)}$ is one-dimensional and is spanned by the vector $\mathbf{s}_{S,{\widetilde\bfr},{\widetilde\bfj}}$.

Otherwise, we can still use Remark~\ref{rem:chopoff}(2) to delete the last $d-k-1$ columns as well as any rows outside of $S$ and reduce to the case when $k=n=d-1$. In this case, we have $Z_{(S,\widetilde\bfr,\widetilde\bfj)}=0$ by Lemma~\ref{lem:knd-1}, since it implies that the $(k+1)^{th}$ column can not have a decoration in a row indexed by $S$.

Thus, we conclude that the vector $\mathbf{s}_{S,\widetilde{\bfr},\widetilde{\bfj}}$ spans the space $Z_{(S,\widetilde{\bfr},\widetilde{\bfj})}$, whenever it is non-zero, completing the proof of the lemma.
\end{proof}

Since the action of $\ell\in MU(n)$ commutes with that of $MH_d$, the space $Z_k$ is invariant under the action of $\ell$. Furthermore, the $1$-eigenspace of $\ell$ in $Z_k$ is $\ell\cdot Z_k$, since $\ell$ is an idempotent.

\begin{corollary} \label{cor:zerodim}
We have the following equalities:
\begin{equation} \label{eq:zerodim}
\dim(Z_k)=\sum_{t=0}^{\min(n,d)-k}\binom{n}{k+t}\binom{d-k}{t}n^{d-k-t},\end{equation}
 \begin{equation} \dim(\ell\cdot Z_k)=\sum_{t=0}^{\min(n,d)-k}\binom{n-1}{k+t}\binom{d-k}{t}n^{d-k-t}.\end{equation}
\end{corollary}

\begin{proof}
By Lemma~\ref{lem:zerobasis}, the space $Z_k$ has a basis given by elements of the form $\mathbf{s}_{S,{\mathbf{r}},{\mathbf{j}}}$ where $S\in [n]_k$, ${\mathbf{r}}=(r_{k+1},r_{k+2},\dots, r_{d})\in\{1,2,\dots,n\}^{d-k}$ and ${\bfj}\in J_{{\bfr}}(S)$. To count the number of these triples $(S,\bfr,\bfj)$, suppose we assume that $|{\bfj}|=t$ for some $0\leq t\leq \min\{n,d\}-k$. Then, we can select the subset $S$ as well as the $t$ rows that are being decorated in $\binom{n}{k+t}$ ways and the $t$ columns being decorated in ${d-k\choose t}$ ways. The remaining $d-t-k$ entries of ${\bfr}$ can be chosen arbitrarily. Hence, summing over all $t$, we get:
\[\dim(Z_k)=\sum_{t=0}^{\min(n,d)-k}{n \choose k+t}{d-k \choose t}n^{d-k-t},\]
proving equality~\eqref{eq:zerodim}.

Next, we analyze the $\ell$-action. If $1\in S$, we must have ${\bfj}=\emptyset$, and so, $\ell\cdot \mathbf{s}_{S,{\mathbf{r}},{\mathbf{j}}}=0$ as a consequence of Corollary~\ref{cor:lconse}, just as in the proof of Lemma~\ref{lem:simzero}. Now, suppose $1\not\in S$. For any triple $(S,\bfr,\bfj)$ such that $1\not\in S$ and $r_j=1$ for some $j\in \bfj$, we have by Corollary~\ref{cor:lconse} that $\ell\cdot \mathbf{s}_{S,{\mathbf{r}},{\mathbf{j}}}$ is a non-zero scalar multiple of $\ell\cdot \mathbf{s}_{S,{\mathbf{r}},{\mathbf{j}}\sm \{j\}}.$ Therefore, a spanning set for the space $\ell\cdot Z_k$ is given by $\ell\cdot \mathbf{s}_{S,{\mathbf{r}},{\mathbf{j}}}$ for those triples $(S,{\bfr},{\bfj})$ such that $1\not\in S$ and for all $j\in {\bfj}$, we have $r_j\neq 1$. Furthermore, these elements are linearly independent as they lie in the span of different basis elements $m_{\widetilde\bfr,\widetilde\bfj}$ in $MV_{n,d}$, which is a consequence of Equation~\eqref{eq:laction}:
\[\ell\cdot \mathbf{s}_{S,{\mathbf{r}},{\mathbf{j}}}=v^{-2\#\{i> \max\bfj~|~r_i=1\}}\left(\mathbf{s}_{S,{\mathbf{r}},{\mathbf{j}}}+\sum\limits_{t>\max\bfj,~r_t=1}\mathbf{s}_{S,{\mathbf{r}},{\mathbf{j}}\cup\{t\}}\right).\]
Therefore, the dimension of the space $\ell\cdot Z_k$ is exactly the number of such triples. This number can be computed by repeating the computation we did for $\dim(Z_k)$ above after replacing ${n\choose k+t}$ by ${n-1 \choose k+t}$ to account for the fact that the first row shouldn't be included in $S$ and shouldn't be decorated in $\bfj$, thus completing the proof.
\end{proof}

\subsection{Proof of the base case $s=\min(n,d)$ of Theorem~\ref{thm:MSWD}.} \label{sec:base}
If $s=d\leq n$, by Lemma~\ref{lem:zerobasis}, there exist non-zero vectors in $MV_{n,d}$ on which all the Jucys-Murphy elements $\LL_1,\ldots,\LL_d$ act by zero. These vectors each span one-dimensional $MH_d$-representations corresponding to the bipartition $(\emptyset,1^d)$. By Proposition~\ref{prop:leftmod}, the only $L_{\bs\lambda}$'s that pair with a one-dimensional representation of $MH_d$ correspond to $\bs\lambda=(d,\emptyset)$, $\bs\lambda=(1^d,\emptyset)$ or $\bs\lambda=(\emptyset,1^d)$. In the first two cases, each of the $\LL_i$'s act with non-zero eigenvalues by Proposition~\ref{prop:s=0}. Thus, the representation $M^{(\emptyset, 1^d)}$ of $MH_d$ must pair with the $L_{\bs\lambda}$ when $\bs\lambda=(\emptyset,1^d)$, proving the claim.

Next, suppose $s=n<d$ and $\bs\lambda=(\lambda,1^n)$ so that $|\lambda|+n=d$. Then by Lemma~\ref{lem:zerobasis}, the space $Z_n$ is spanned by elements of the form $\mathbf{s}_{S,{\mathbf{r}},{\mathbf{j}}}$ with $S=\{1,2,\dots,n\}$, arbitrary ${\bfr}$ chosen in $\{1,2,\dots,n\}^{d-n}$ and ${\bfj}=\emptyset$. Lemma~\ref{lem:simzero} (along with Remark~\ref{rem:s-lem41}) shows that the idempotent $\ell$ acts by zero on each of these elements. This shows that, as a left $MU(n)$-module, the space $Z_n$ lies in the span of the isotypic components of those $L_{\bs\lambda}$'s on which $\ell$ acts by zero, that is, those $\bs\lambda$'s that have $s=n$.

For any $\lambda$ with $|\lambda|=d-n$ and a tableau $U$ of shape $\lambda$, by Theorem~\ref{thm:QSWD}, there exists an element in the quantum tensor space $V_{n,d-n}$:
\[s_{\lambda}=\sum_{\mathbf{r}\in \Pi} c_{\mathbf{r}}m_{\mathbf{r}},\]
which corresponds to the tableau $U$ with respect to the $H_{d-n}$-action and the quantum group $\uvsln$ acts on it with weight $\lambda$. (Here, $\Pi\sub\{1,2,\dots,n\}^{d-n}$ is some index set and $c_{\mathbf{r}}$ are some scalars.) That is, $s_{\lambda}=w_{\lambda}\otimes x_U \in L_{\lambda}\otimes M^{\lambda}\sub V_{n,d-n}$, where $w_{\lambda}$ is the unique (up to scaling) highest weight vector in $L_{\lambda}$ and $x_U$ is the basis element of $M^{\lambda}$ mentioned in Remark~\ref{rem:hecke-basis}.  Then, we can construct:
\[s_{\bs\lambda}:=\sum_{\mathbf{r}\in \Pi} c_{\mathbf{r}}\mathbf{s}_{S,\mathbf{r},\emptyset},\]
where $S=\{n,n-1,\dots,1\} \in [n]_n$. 

Consider the linear map:
\begin{align*}
\pi_{n,d}:V_{n,d-n}&\to MV_{n,d}\\
m_{\mathbf{r}}&\mapsto \mathbf{s}_{S,\mathbf{r},\emptyset}
\end{align*}
which intertwines the right action of $\tau_{i-n}\in H_{d-n}$ on the domain with the action of $\tau_{i}\in MH_d$ on the co-domain for $n+1\leq i\leq d-1$, as there are no decorations in the last $d-n$ columns of any summand of $\mathbf{s}_{S,\mathbf{r},\emptyset}$. Note that $\pi_{n,d}(s_{\lambda})=s_{\bs\lambda}$.

By Lemma~\ref{lem:zerobasis}, the elements $\LL_1,\LL_2,\dots,\LL_n$ act on $s_{\bs\lambda}$ by zero, whereas $\LL_{n+1}$ acts on it by $1$. 
Furthermore, for $n+2\leq j\leq d$, by the equivariance of the map $\pi_{n,d}$ mentioned above, the Jucys-Murphy element $\LL_j \in MH_{d}$ acts on $s_{\bs\lambda}$ by the same eigenvalue as the Jucys-Murphy element $\LL_{j-n} \in H_{d-n}$ acts on $s_{\lambda}$. Therefore, all the Jucys-Murphy elements of $MH_d$ act on $s_{\bs\lambda}$ by exactly the same eigenvalues as they do on the basis element $x_{(U_1,U_2)}\in MV^{\bs\lambda}$, where $U_1$ is obtained from $U$ by increasing each entry by $n$ and $U_2$ is the one-column tableau with entries $\{1,2,\dots, n\}$ increasing from top to bottom. Hence, $s_{\bs\lambda}$ lies in the right isotypic component of the simple $MH_d$-representation $M^{\bs\lambda}$.

Recall from Equation~\eqref{eq:kaction} that $k_i$ acts on $m_{\bfr,\bfj}$ by the eigenvalue $v^{\#\{j:r_j=i\}-\#\{j:r_{j}=i+1\}}$. Therefore, since every summand in $\mathbf{s}_{S,\mathbf{r},\emptyset}$ has exactly one more $1$ in each row as $m_{\bfr}$, the $k_i$'s act on $\mathbf{s}_{S,\mathbf{r},\emptyset}$ with the same weight as they do on $m_{\bfr}$. Hence, since the vector $s_{\lambda}$ has weight $\lambda$ with respect to the quantum group $\uvsln$, we must also have that $s_{\bs\lambda}$ has weight $\lambda$ with respect to $MU(n)$. Furthermore, since $\ell$ acts on it by zero, the element $s_{\bs\lambda}$ lies in the left isotypic component of the $MU(n)$-representation $L_{\bs\lambda'}$ for some $\bs\lambda'=(\lambda',1^n)$ such that $\lambda'\geq \lambda$.

Thus, we conclude that for all $\bs\lambda=(\lambda,1^n)\in M\Lambda_{n,d}$, the isotypic component of the $MH_d$-representation $M^{\bs\lambda}$ in $MV_{n,d}$ is non-zero. Furthermore, if the representation $M^{\bs\lambda}$ pairs with $L_{\bs\lambda'}$ in the mirabolic tensor space, then $\bs\lambda'=(\lambda',1^n)$ for some $\lambda'\geq \lambda$. Using induction on $\lambda$ with respect to the partial order on the set of partitions of $d-n$, we conclude that $M^{\bs\lambda}$ pairs exactly with $L_{\bs\lambda}$ for all $\bs\lambda$ when $s=d$, completing the base case of the induction.

\subsection{Induction step}  \label{sec:ind}

Fix a bipartition $\bs\lambda=(\lambda,1^s)$. We assume that Theorem~\ref{thm:MSWD} has been proven whenever the second partition has more than $s$ parts. That is, we assume that $L_{\bs\lambda'}$ pairs with $M^{\bs\lambda'}$ in the mirabolic tensor space whenever $\bs\lambda'=(\lambda',1^{s'})$ for $s'>s$.

Recall the subspace $Z_s$ of all elements in the mirabolic tensor space on which the first $s$ Jucys-Murphy elements act by zero. Note that since the $MU(n)$-action commutes with the $MH_d$-action, the space $Z_s$ has a left $MU(n)$-action. We can write $Z_s=Z\oplus Z'$, where $Z$ is the intersection of $Z_s$ with the isotypic components of those $L_{\bs\lambda'}$'s having depth at most $s$ and $Z'$ is the intersection of $Z_s$ with the isotypic components of those $L_{\bs\lambda'}$'s (and by induction, those $M^{\bs\lambda'}$'s) whose depth is bigger than $s$. We remind the reader that the simple $MU(n)$-module $L_{(\lambda,1^r)}$ has depth exactly $r$ (see Example~\ref{ex:(co-)depth}).

\begin{lemma} \label{lem:intcompu}
\begin{enumerate}
\item The space $Z$ is the intersection of $Z_s$ with the isotypic components of those $L_{\bs\lambda'}$'s whose depth is exactly $s$.
\item For any $k\geq s+1$, we have the equality $\ker(\LL_k|_{Z})=0$. In particular, the Jucys-Murphy element $\LL_{s+1}$ acts on $Z$ by 1.
\end{enumerate}
\end{lemma}

\begin{proof}
($1$) Fix $\bs\lambda'=(\lambda',1^{s'})$ for some $s'>s$. The space $M^{\bs\lambda'}$ has a basis given by $x_U$ where $U\in T(\bs\lambda')$, such that exactly $s'$ of the $d$ Jucys-Murphy elements act by zero on any $x_U$. (The last assertion follows from Theorem~\ref{thm:JM_action}.) Therefore, the dimension of the space $\cap_{i=1}^s \ker(\LL_i|_{M^{\bs\lambda'}})$ is equal to
\[\frac{{d-s\choose s'-s}}{{d\choose s'}}\dim(M^{\bs\lambda'})\Big(={d-s\choose s'-s}\dim(M^{\lambda'})\Big).\]

The dimension of the space $Z'$ is equal to the sum of dimensions of the simultaneous eigenspaces $\cap_{i=1}^s \ker(\LL_i)$ in the isotypic components of those $M^{\bs\lambda'}$'s such that $s'>s$.  By induction, the isotypic component of such an $M^{\bs\lambda'}$ is equal to $L_{\bs\lambda'}\otimes M^{\bs\lambda'}$. Therefore, we have that (all the sums below are over partitions with at most $n$ parts):
\begin{align*}
\dim(Z')&=\sum_{s'=s+1}^{\min(n,d)}\dfrac{\binom{d-s}{s'-s}}{\binom{d}{s'}}\sum_{\substack{\bs\lambda'=(\lambda',1^{s'})\\ |\lambda'|=d-s'}} \dim(L_{\bs\lambda'})\dim(M^{\bs\lambda'})\\
&=\sum_{t=1}^{\min(n,d)-s}\frac{{d-s\choose t}}{{d\choose s+t}}\sum_{\substack{\bs\lambda'=(\lambda',1^{s+t})\\ |\lambda'|=d-s-t}} \dim(L_{\bs\lambda'})\dim(M^{\bs\lambda'})&(\textrm{where }t=s'-s)\\
&=\sum_{t=1}^{\min(n,d)-s}\frac{{d-s\choose t}}{{d\choose s+t}}\sum_{\substack{|\lambda'|=d-s-t}}\dim(L_{\lambda'})\dim(M^{\lambda'}){n\choose s+t}{d\choose s+t}\\
&=\sum_{t=1}^{\min(n,d)-s}{n\choose s+t}{d-s\choose t}n^{d-s-t},
\end{align*}

where the last equality follows from Theorem~\ref{thm:QSWD} and the fact that $\dim(V_{n,d-s-t})=n^{d-s-t}$.

Therefore, by Corollary~\ref{cor:zerodim}, the dimension of $Z$ is $\dim(Z_s)-\dim(Z')={n \choose s}n^{d-s}$.
Next, by a computation very similar to the one above, we can compute:
\[\dim(\ell\cdot Z')=\sum_{t=1}^{\min(n,d)-s}{n-1\choose s+t}{d-s\choose t}n^{d-s-t}.\]
(The only step that changes in the computation is that we need to replace the isotypic component $L_{\bs\lambda'}\otimes M^{\bs\lambda'}$ with its $\ell=1$ eigenspace. That's the same as multiplying the summand of the computation above by the fraction ${n-1 \choose s+t}/{n \choose s+t}$, since $L_{\bs\lambda'}$ has depth $s'=s+t$.)

Therefore, by Corollary~\ref{cor:zerodim}, the dimension of the $1$-eigenspace of $\ell$ in $Z$ is ${n-1 \choose s}n^{d-s}$. We observe that:
\[\frac{\dim(\ell\cdot Z)}{\dim(Z)} = \frac{{n-1 \choose s}n^{d-s}}{{n \choose s}n^{d-s}}=\frac{n-s}{n}.\]
Note that for any $L_{\bs\lambda'}$ for which $s'\leq s$, we have 
\[\frac{\dim(\ell\cdot L_{\bs\lambda'})}{\dim(L_{\bs\lambda'})}=\frac{n-s'}{n}\geq \frac{n-s}{n}.\]

Since $Z$ lies in the span of such $L_{\bs\lambda'}$'s, the above equality along with the \emph{mediant inequality} implies that $Z$ lies in the span of those $L_{\bs\lambda'}$'s whose depth is exactly $s$.

($2$) Note that we have inclusions:
\[Z'\sub Z_s \cap \Big(\sum_{t=s+1}^d \ker(\LL_t)\Big) \sub Z_s.\]
We prove the required claim by showing that the first inclusion is an equality: this would imply that $\ker(\LL_t|_{Z_s})=\ker(\LL_t|_{Z'})$ and since $(Z)\LL_t\subset Z$, $(Z')\LL_t\subset Z'$, we have that $\ker(\LL_t|_{Z_s})=\ker(\LL_t|_{Z})\oplus \ker(\LL_t|_{Z'})$, so we would indeed get that $\ker(\LL_t|_{Z})=0$. 

We show this by proving that the dimensions of the two spaces are equal. To compute the dimension of the space $Z_s \cap \Big(\sum_{t=s+1}^d \ker(\LL_t)\Big) = \bigcap_{t=1}^s \ker(\LL_t) \cap \Big(\sum_{t=s+1}^d \ker(\LL_t)\Big)$, we use inclusion-exclusion. This is justified by Remark~\ref{rem:simulbases}(1), which says that performing inclusion-exclusion on the dimensions of these spaces is equivalent to performing set-theoretic inclusion-exclusion on their simultaneous bases.

Therefore, we have
\begin{align*}
&\dim\Bigg(\bigcap_{t=1}^s \ker(\LL_t) \cap \Big(\sum_{t=s+1}^d \ker(\LL_t)\Big)\Bigg)\\
=&\sum_{k=1}^{\min(n,d)-s}\sum_{s+1\leq i_1 < i_2<\dots<i_k\leq d} (-1)^{k-1} \dim\Big(\bigcap_{t=1}^s \ker(\LL_t) \cap\bigcap_{t=1}^k \ker(\LL_{i_t})\Big)\\
=&\sum_{k=1}^{\min(n,d)-s}\sum_{s+1\leq i_1 < i_2<\dots<i_k\leq d} (-1)^{k-1} \dim\Big(\bigcap_{t=1}^{s+k} \ker(\LL_t)\Big)\\
=&\sum_{k=1}^{\min(n,d)-s}(-1)^{k-1}{d-s\choose k} \dim\Big(\bigcap_{t=1}^{s+k} \ker(\LL_t)\Big)\\
=&\sum_{k=1}^{\min(n,d)-s}(-1)^{k-1}{d-s\choose k} \sum_{t=0}^{\min(n,d)-s-k}{n \choose s+k+t}{d-s-k \choose t}n^{d-s-t-k} &(\textrm{By Equation~\eqref{eq:zerodim}})\\
=&\sum_{l=1}^{\min(n,d)-s}{n \choose s+l}n^{d-s-l}\sum_{k=1}^{l}(-1)^{k-1}{d-s\choose k}{d-s-k \choose l-k} &(\textrm{where }l=k+t)\\
=&\sum_{l=1}^{\min(n,d)-s}{n \choose s+l}{d-s\choose l}n^{d-s-l} 
\\
=&\dim(Z').
\end{align*}
The second equality above holds since the simultaneous zero eigenspaces of any $s+k$ Jucys-Murphy elements have the same dimension. The expression for $\dim(Z')$ was computed in the proof of the first part of the Lemma.
\end{proof}

Next, we construct a subspace $\mathcal{Z}$ of $Z_s$ as follows.
Let $S=\{n,n-1,\dots,n-s+1\}\in [n]_s$. Fix arbitrary ${\mathbf{r}}\in\{1,2,\dots,n\}^{d-s}$. Define:
\[\mathbf s_{S,{\mathbf{r}}} = \sum_{{\mathbf{j}}\in J_{\bfr}(n-s+1)}\mathbf s_{S,{\mathbf{r}},{\mathbf{j}}}.\]
Recall that $J_{\bfr}(n-s+1)$ is the set of decorations ${\mathbf{j}}\in J_{\bfr}$ such that $a\in {\mathbf{j}}$ only if $r_a\leq n-s$.

We define $\mathcal{Z}:=\operatorname{Span}\{\mathbf s_{S,{\mathbf{r}}}~|~\bfr\in\{1,2,\dots,n\}^{d-s}\}$). As these elements are all linearly independent, we have that $\dim(\mathcal{Z})=n^{d-s}$. Furthermore, by Lemma~\ref{lem:zerobasis}, the first $s$ Jucys-Murphy elements act by zero on $\mathcal{Z}$.

Recall the map $\phi_{n,d}^{k,t}:MV_{n,d}^{k,t}\to MV_{n,d}$ defined via Equation~\eqref{eq:phindkt} for $0\leq k\leq n$ and $0\leq t\leq d$. We consider the case when $k=n-s$ and $t=s$. By \eqref{eq:phi-kt} we have that $\mathbf s_{S,{\mathbf{r}}}=\phi_{n,d}^{n-s,s}(\mathbf s_{S,\bfr,\emptyset})\in\im(\phi_{n,d}^{n-s,s})$ for all $\bfr$. Therefore, by Corollary~\ref{lem:averowssl}, the elements $\mathbf s_{S,{\mathbf{r}}}$ have co-depth at least $n-s$. Therefore, the space $\mathcal{Z}$ lies in the span of the isotypic components of those $L_{\bs\lambda'}$'s that have co-depth at least $n-s$ (or depth at most $s$), and thus, $\mathcal{Z}\sub Z$. Then, by Lemma~\ref{lem:intcompu}(1), the space $\mathcal{Z}$ must lie in the span of the isotypic components of those $L_{\bs\lambda'}$'s that have depth exactly $s$, and thus by Lemma~\ref{lem:oldepth}(1), the elements $\mathbf s_{S,{\mathbf{r}}}$  must have co-depth exactly $n-s$. As $\mathcal{Z}$ can't have any non-zero elements with a higher co-depth, each non-zero linear combination of the  $\mathbf s_{S,{\mathbf{r}}}$'s also has co-depth $n-s$. Lastly, Lemma \ref{lem:intcompu}(2) shows that the Jucys-Murphy element $\LL_{s+1}$ acts on the space $\mathcal{Z}$ by 1.

Now we finish the induction step of the proof of Theorem~\ref{thm:MSWD} using an argument fairly similar to the one used in the proof of the base case (see Section~\ref{sec:base}). For any $\lambda$ with $|\lambda|=d-s$ and at most $n$ parts, we consider $\bs\lambda=(\lambda,1^s)\in M\Lambda_{n,d}$. Given a tableau $U$ of shape $\lambda$, by Theorem~\ref{thm:QSWD}, there exists an element in the quantum tensor space $V_{n,d-s}$:
\[s_{\lambda}=\sum_{\mathbf{r}\in \Pi} c_{\mathbf{r}}m_{\mathbf{r},},\]
which corresponds to the vector $x_U\in M^{\lambda}$ with respect to the $H_{d-s}$-action and the quantum group $\uvsln$ acts on it with weight $\lambda$. (Here, $\Pi\sub\{1,2,\dots,n\}^{d-n}$ is some index set and $c_{\mathbf{r}}$ are some scalars.) More precisely, this is exactly the element $w_{\lambda}\otimes x_U\in V_{n,d-n}$ considered in Section~\ref{sec:base}. Then, we can define:
\[s_{\bs\lambda}:=\sum_{\mathbf{r}\in \Pi} c_{\mathbf{r}}\mathbf{s}_{S,\mathbf{r}}=\sum_{\mathbf{r}\in \Pi} c_{\mathbf{r}}\phi_{n,d}^{n-s,s}(\mathbf s_{S,\bfr,\emptyset})\in \mathcal{Z}.\]
Consider the linear map:
\[\pi_{s,n,d}:V_{n,d-s}\to MV_{n,d}\]
\[m_{\bfr}\mapsto \mathbf{s}_{S,\bfr,\emptyset}.\]
This map $\pi_{s,n,d}$ intertwines the right action of $\tau_{i-s}\in H_{d-s}$ on the domain with the action of $\tau_i\in MH_d$ on the co-domain, as there are no decorations in the last $d-s$ columns of any summand of $\mathbf{s}_{S,\mathbf{r},\emptyset}$. We note that $s_{\bs\lambda}=\phi_{n,d}^{n-s,s}\circ \pi_{s,n,d}(s_\lambda).$

Then, by an argument similar to the one in Section~\ref{sec:base}, the Jucys-Murphy elements of the mirabolic Hecke algebra $MH_d$ act on $s_{\bs\lambda}$ exactly by the same eigenvalues as they do on the element $x_{(U_1,U_2)}\in M^{\bs\lambda}$, where $U_1$ is obtained from $U$ by increasing each entry by $s$ and $U_2$ is the one-column tableau with entries $\{1,2,\dots, s\}$ increasing from top to bottom. This is because $\LL_k$ with $1\leq k\leq s$ all act by zero and $\LL_{s+1}$ acts by $1$ (as $s_{\bs\lambda}\in \mathcal{Z}$), and for the rest we use the equivariance of $\phi_{n,d}^{n-s,s}$ for the action of $\tau_j$, $j\neq s$ from Lemma~\ref{lem:averows} and the above intertwining property of $\pi_{s,n,d}$.

Next, with respect to the $MU(n)$-action, the vector $s_{\bs\lambda}$ has co-depth $n-s$ (since it is a non-zero linear combination of the $\mathbf{s}_{S,\bfr}$'s) and lies in the isotypic component of $L_{\bs\lambda'}$ for some $\bs\lambda'=(\lambda',1^s)$ (since it is an element of $\mathcal{Z}$). Furthermore, with respect to the $MU(n)$-action, the vector $s_{\bs\lambda}$ has weight $\lambda +\omega_s'$, where $\omega_s'=(0,\dots,0,1,\dots,1)=(0^{n-s}1^s)$. This follows from Equation~\eqref{eq:kaction} describing the $k_i^{\pm}$-action and the fact that when $\mathbf{s}_{\bs\lambda}$ is expressed in terms of the decorated matrix basis $\{m_{\bfr,\bfj}\}$, each non-zero summand has exactly one more $1$ in the bottom $s$ rows as the ones in the expression of $s_{\lambda}$.

By Lemma~\ref{lem:oldepth}(2), there exists a vector $s_{\bs\lambda}'$ with the same $\LL_i$-eigenvalues as $s_{\bs\lambda}$ such that $MU(n)$ acts on $s_{\bs\lambda}'$ with weight $\lambda+\omega_s$, where $\omega_s=(1,\dots,1,0,\dots,0)=(1^{s}0^{n-s})$. Therefore, the vector $s_{\bs\lambda}'$ must lie in the isotypic component of the $MH_d$-representation $M^{\bs\lambda}$ and the $MU(n)$-representation $L_{
\bs\lambda'}$ for some $\bs\lambda'=(\lambda',1^s)$ such that $\lambda'\geq \lambda$.

Using induction on $\lambda$ with respect to the partial order on the set of partitions of $d-s$, we conclude that $M^{\bs\lambda}$ pairs with $L_{\bs\lambda}$ for all $\bs\lambda$, completing the proof.

\bibliographystyle{alpha}
\bibliography{mirabolic-sln-biblist} 

\end{document}